\documentclass{amsart}

\usepackage{tikz}
\usepackage{tikz-cd}
\usepackage{hyperref}
\usepackage{amssymb}
\usepackage{amsmath}
\usepackage{amsthm}

\makeatletter
\def\l@section{\@tocline{2}{3pt}{1pc}{5pc}{}}
\def\l@subsection{\@tocline{2}{0pt}{1.5pc}{5pc}{}}
\makeatother

\usepackage[style=alphabetic, maxnames=4]{biblatex}
\newtheorem{thm}{Theorem}[subsection]
\newtheorem{thm*}{Theorem}
\newtheorem{prop}[thm]{Proposition}
\newtheorem{cor}[thm]{Corollary}
\newtheorem{lem}[thm]{Lemma}

\newtheorem{athm}{Theorem}[section]
\newtheorem{aprop}[athm]{Proposition}
\newtheorem{acor}[athm]{Corollary}
\newtheorem{alem}[athm]{Lemma}

\theoremstyle{definition}
\newtheorem{defi}[thm]{Definition}
\newtheorem{defi*}[thm*]{Definition}
\newtheorem{ex}[thm]{Example}

\newtheorem{rem}[thm]{Remark}
\newtheorem{war}[thm]{Warning}
\newtheorem{con}[thm]{Construction}
\newtheorem{quest}[thm]{Question}

\newtheorem{adefi}[athm]{Definition}
\newtheorem{adefi*}{Definition}
\newtheorem{aex}[athm]{Example}

\newtheorem{arem}[athm]{Remark}

\newtheorem{acon}[athm]{Construction}

\newtheorem*{ack}{Acknowledgements}

\newcommand{\An}{\text{\normalfont An}}
\newcommand{\Ring}{\text{\normalfont Ring}}
\newcommand{\Ab}{\text{\normalfont Ab}}

\newcommand{\Sch}{\text{\normalfont Sch}}
\newcommand{\dSch}{\text{\normalfont dSch}}
\newcommand{\CAlg}{\text{\normalfont CAlg}}

\newcommand{\colim}{\text{\normalfont colim}}
\renewcommand{\lim}{\text{\normalfont lim}}
\newcommand{\Sp}{\text{\normalfont Sp}}

\newcommand{\Spec}{\text{\normalfont Spec}\,}
\newcommand{\Open}{\text{\normalfont Open}}

\newcommand{\Free}{\text{\normalfont Free}}

\newcommand{\Freep}{\text{\normalfont Free}_{\delta_p}}
\newcommand{\Cofree}{\text{\normalfont Cofree}_{\hat\delta}}
\newcommand{\Forget}{\text{\normalfont Forget}}
\newcommand{\Map}{\text{\normalfont Map}}

\newcommand{\comp}[1]{#1_p^{\wedge}}
\newcommand{\can}{\text{\normalfont can}}
\newcommand{\Frob}{\text{\normalfont Frob}}
\newcommand{\an}{\text{\normalfont an}}

\newcommand{\fib}{\text{\normalfont fib}}

\newcommand{\cn}{\text{\normalfont cn}}
\newcommand{\op}{\text{\normalfont op}}
\newcommand{\alg}{\text{\normalfont alg}}

\newcommand{\tCp}{tC_p}
\newcommand{\PrL}{\text{\normalfont Pr}^{\text{\normalfont L}}}

\newcommand{\Hom}{\text{\normalfont Hom}}

\newcommand{\Der}{\text{\normalfont Der}}

\newcommand{\Mod}{\text{\normalfont Mod}}
\newcommand{\dis}{\text{\normalfont dis}}
\newcommand{\Zar}{\text{\normalfont Zar}}
\newcommand{\aff}{\text{\normalfont aff}}
\newcommand{\Str}{\text{\normalfont Str}}

\newcommand{\A}{\mathbb A}
\newcommand{\E}{\mathbb E}

\newcommand{\G}{\mathbb G}
\renewcommand{\P}{\mathbb P}
\renewcommand{\S}{\mathbb S}
\newcommand{\Z}{\mathbb Z}
\newcommand{\N}{\mathbb N}
\newcommand{\Q}{\mathbb Q}
\renewcommand{\O}{\mathcal O}

\newcommand{\D}{\mathcal D}

\newcommand{\p}{\mathfrak p}

\title{An obstruction to lifting schemes to spectral schemes}
\author{Robert Szafarczyk}
\date{}

\begin{document}

\begin{abstract}
\(\hat\delta\)-rings are a globalization of \(\delta\)-rings. Motivated by their application to obstructing spectral lifts, we develop the theory of animated \(\hat\delta\)-rings. They extend the discrete notion in the same way that animated \(\delta\)-rings extend that of \(\delta\)-rings, but are not given by the process of animation. We also generalize the definition to schemes and extend the obstruction to this setting. We apply it to concrete examples such as number rings, closed subschemes of \(\P^n\), and various group schemes.
\end{abstract}

\maketitle

\tableofcontents

\section{Introduction}

Several structures in algebra have been set up with motivation of capturing residual data coming from a more complicated object. For example, \(\lambda\)-rings (\cite{lambda}) have been extensively studied, because they axiomatize the structure of Adams operations on algebraic \(K\)-theory of a commutative ring. The operations themself come from the action of polynomial functors (e.g.\@ exterior powers) on the module category (see \cite{polynomial} for the strongest version of the result).

From a different perspective, \(\lambda\)-rings encode pairwise commuting integral Frobenius lifts over every prime number (\cite{Wilkerson}). Restricting to a single prime \(p\), one arrives at the notion of a \(\delta_p\)-ring (\cite{Joyal}):

\begin{defi*} \label{def_dfun}
Let \(R\) be a discrete ring and \(p\) a prime number. A \(\delta_p\)-structure on \(R\) is a function \(\delta:R\to R\) on underlying sets satisfying:
\begin{enumerate}
\item \(\delta(0)=\delta(1)=0\)
\item \(\delta(x+y)= \delta(x)+\delta(y)+\frac{x^p+y^p-(x+y)^p}p\)
\item \(\delta(xy)=x^p\delta(y)+\delta(x)y^p+p\delta(x)\delta(y)\).
\end{enumerate}
\(\phi(x)=x^p+p\delta(x)\) then defines an integral lift of the Frobenius endomorphism of~\(R/p\).
\end{defi*}

Such structures are also closely related to cohomology theories in algebraic geometry where it becomes useful to extend them to animated rings (the homotopy theory of simplicial rings) in order to capture more refined information and gain access to the powerful tools of homotopy theory. For \(\delta\)-rings this was done in \cite{Prismatization} with applications to prismatic cohomology and for \(\lambda\)-rings in \cite{Hubner}.

In this paper, we extend to the animated setting the notion of a \(\hat\delta\)-ring (read:\@ delta-hat ring) introduced in \cite{group_rings}:

\begin{defi*}
An (animated) \(\hat\delta\)-ring is an (animated) ring \(R\) together with an (animated) \(\delta_p\)-structure on its \(p\)-completion \(\comp R\) for all prime numbers \(p\).
\end{defi*}

In contrast to \(\lambda\)-rings and \(\delta\)-rings, the category of animated \(\hat\delta\)-rings is not the animation of a \(1\)-category. Moreover, the forgetful functor to animated rings does not admit a left adjoint, so it is not monadic. Truncations are not underlying either. We discuss this in Section \ref{sec_hatrings}.

The motivation for \(\hat\delta\)-rings comes from the following theorem.

\begin{thm*}[{\cite[Proposition 2.41]{group_rings}}] \label{thm_intro1}
The base change functor
\[-\otimes_{\S} \Z:\CAlg_{\S}\to\CAlg_\Z,\]
restricted to the full subcategory of those connective \(\E_\infty\)-rings with homology concentrated in degree \(0\), refines to a functor taking values in the category of \(\hat\delta\)-rings.
\end{thm*}

By this result, existence of \(\hat\delta\)-structures becomes an obstruction for lifting discrete rings to the sphere spectrum. Here, by a lift of a ring \(R\) we mean a connective \(\E_\infty\)-ring \(\S_R\) satisfying \(\S_R\otimes_\S\Z\simeq R\).

The lifting problem has already been studied in the literature and is equivalent to searching for \(\E_\infty\)-structures on Moore spectra such as \(\S[i]\) (of \(\Z[i]\)), \(\S/n\) (of \(\Z/n\)), \(\S_E\) (of an étale algebra \(E\)), \(\S[x]\) (of \(\Z[x]\)), and so forth.

The case of \(\S[i]\), and adjoining roots of unity to \(\E_\infty\)-ring spectra in general, was first considered in \cite{roots}. Using base change formulas for topological Hochschild homology the authors disproved the existence of such an \(\E_\infty\)-ring structure. In \cite[Remark 4.3]{moore}, using tools of chromatic homotopy theory, it was shown that \(\S/n\) is never \(\E_\infty\).

On the other hand, any étale algebra \(E\) admits a unique \(\E_\infty\)-lift \(\S_E\) by the celebrated result of Lurie (\cite[Theorem 7.5.0.6]{HA}) and \(\S[x]\) can be constructed as the spherical monoid ring \(\S[\N]=\Sigma^\infty_+\N\).

In each of the aforementioned examples very different techniques were employed to obstruct multiplicative structures on spherical lifts. Theorem \ref{thm_intro1} is fully functorial and gives a relatively simple algebraic obstruction for lift existence.

In the paper we will apply it in several examples. Non-existence of \(\E_\infty\)-structures on \(\S/n\) becomes a quick calculation (see Corollary \ref{cor_smodn}). We look at rings of integers and étale algebras as well. For example, we show that no ring of integers of a number field lifts to \(\S\) as an \(\E_\infty\)-algebra and prove the same for orders of the form \(\Z[ni]\) (see Chapter \ref{chap_ex}).

Another contribution of our work is extending the obstruction and all relevant notions to schemes.

\begin{defi*}
A \emph{\(\hat\delta\)-scheme} is a topological space \(X\) together with a sheaf of \(\hat\delta\)-rings \(\O_X^{\hat\delta}:\Open(X)^\op\to\Ring^\an_{\hat\delta}\) such that \(X\) admits a covering \(\{U_i\}\) for which \((U_i,\O_X^{\hat\delta}|_{U_i})\) is equivalent to an affine \(\hat\delta\)-scheme.
\end{defi*}

\begin{thm*}[Theorem \ref{thm_schemes}]
The base change functor
\[-\times_{\Spec \S}\Spec \Z: \Sch_{\S}\to \Sch_\Z,\]
restricted to the full subcategory of those connective spectral schemes that are sent to discrete schemes, refines to a functor taking values in \(\hat\delta\)-schemes.
\end{thm*}

Using this we show that neither the additive group scheme \(\G_a\) (see also \cite[Proposition 1.6.20]{ell2}), nor the general linear group \(\text{GL}_n\), for \(n\geq 2\), admit group scheme lifts to \(\Spec\S\). On the other hand, there exists a unique \(\hat\delta\)-group scheme structure on \(\G_m\). We also prove non-liftability for certain closed subschemes of the projective space (see Chapter \ref{chap_ex}).

It is also worth mentioning an unpublished result of Achim Krause that states that \(\hat\delta\)-structures on flat rings are in bijection with \(H_\infty\)-structures on their Moore spectra. Combined with our work this shows Zariski descent for \(H_\infty\)-structures, which is not clear from the definition itself (see Remark \ref{rem_achim}).

There are several open questions one can ask concerning the theory of \(\hat\delta\)-rings. Some of them were already posed in \cite{group_rings}.
\begin{enumerate}
\item Do all \(\hat\delta\)-rings lift to \(\S\)?
\item Does the free \(\delta_p\)-ring lift to \(\S\)?
\item When do morphisms lift? (\cite{group_rings} answers this for group rings)
\item What about lifts of animated rings or along other morphisms than \(\S\to \Z\)?
\item Give more examples (e.g.\@ general orders, other schemes).
\end{enumerate}

In an upcoming work we give a negative answer to question 1. Using a computational approach we find an example of a \(\hat\delta\)-ring that does not lift even to the 1-truncation of the sphere spectrum.

We would also like to bring attention to Section \ref{sec_cotan} of the Appendix. We develop there a strategy for studying \(\delta\)-rings using the cotangent complex formalism. It turned out not to be necessary for our other results, but we think it can be of independent interest. It allows us to reprove all structural properties of animated \(\delta\)-rings of \cite{Prismatization} without the use of Witt vectors.

\begin{ack}
This paper is my master thesis project advised by Thomas Nikolaus. I am very grateful to him for suggesting the topic as well as for our discussions. I would also like to thank him for including me in his research group during my stay in M\"unster, it was an amazing experience. Out of the group I would especially like to thank Edith Hübner for our numerous conversations and Konrad Bals for answering my derived algebraic geometry questions. From Copenhagen, I would like to thank Robert Burklund for several helpful suggestions concerning the introduction.

While in M\"unster, I was supported by the Mathematics M\"unster Short Term Scholarship. Last changes were made when I was supported by Danmarks Grundforskningsfond CPH-GEOTOP-DNRF151.
\end{ack}

\section{\(\hat\delta\)-structures} \label{chap_schemes}

Our main goal in this chapter is to define the notion of a \(\hat\delta\)-scheme. We begin in Section \ref{sec_hatrings} by introducing animated \(\hat\delta\)-rings, the affine building blocks for the theory, and establishing some of their formal properties. Certain aspects, like limits and truncations, do not behave as well as one would wish and require special care. In Section \ref{sec_hatschemes} we establish descent for \(\hat\delta\)-structures and define \(\hat\delta\)-schemes. Our definition is of the form ``a ringed space locally looking like \(\text{Spec}\) of a \(\hat\delta\)-ring".

We work with animated rings (\cite[Section 25.1]{SAG}) because of their good formal behaviour, but an unacquainted reader may suppose all rings to be discrete. The obstruction will only deal with discrete rings afterwards. Nonetheless, it is crucial, even with discrete rings, that we take derived quotients and derived completions in the definitions!

Before we get to the main content, we recall the following definition of \(\delta_p\)-rings that, in the discrete case, is equivalent to Definition \ref{def_dfun}, but readily generalizes to animated rings as well. We refer the reader to the appendix of \cite{Prismatization} for everything related to animated \(\delta_p\)-rings we will use.

\begin{defi*}
Let \(R\) be an animated ring. A \(\delta_p\)-\emph{structure} on \(R\) is an endomorphism \(\phi\) of \(R\) together with a homotopy (in \(\Ring^\an\)) making the diagram
\[\begin{tikzcd}
& R\ar[d,"\can"] \\ R \ar[ur,"\phi"] \ar[r,"\Frob"] & R/p
\end{tikzcd}\]
commute.
\end{defi*}

Hence, a \(\delta_p\)-structure is the same as a derived Frobenius lift. This point of view is much more abstract, which makes it easier to work with. We make use of it in this and the following chapter. We come back to the much more computable Definition \ref{def_dfun} in Chapter \ref{chap_ex}, where we treat examples.

\subsection{Animated \(\hat\delta\)-rings} \label{sec_hatrings}

We introduce one of our main objects of study.

\begin{defi} \label{defi_hatdelta}
An \emph{animated} \(\hat\delta\)-\emph{ring} is an animated ring \(R\) together with a \(\delta_p\)-structure on \(\comp R\) for every prime \(p\). That is, a \(\hat\delta\)-structure on \(R\) consists of morphisms \(\phi:R\to \comp R\), for each prime \(p\), and homotopies making the following diagrams commute:
\[\begin{tikzcd}
& \comp R\ar[d,"\can"] \\ R \ar[ur,"\phi"] \ar[r,"\Frob"] & R/p.
\end{tikzcd}\]
We denote the category of animated \(\hat\delta\)-rings by \(\Ring^\an_{\hat\delta}\) and the full subcategory spanned by discrete rings by \(\Ring_{\hat\delta}\).
\end{defi}

\begin{rem}
If \(R\) is \(p\)-complete, then a \(\hat\delta\)-structure is equivalent to a \(\delta_p\)-structure. If \(R\) is rational, then a \(\hat\delta\)-structure is no data.
\end{rem}

\begin{war}
In contrast to \(\delta_p\)-rings, the category of animated \(\hat\delta\)-rings is not the animation of \(\Ring_{\hat\delta}\). Even worse, there are more discrete objects in \(\Ring^\an_{\hat\delta}\) than those coming from \(\Ring_{\hat\delta}\) (see Proposition \ref{prop_truncated}). Nonetheless, we use the name discrete \(\hat\delta\)-ring to refer to an animated \(\hat\delta\)-ring whose underlying ring is discrete.
\end{war}

\begin{rem}
Even if \(R\) is a discrete ring, its \(p\)-completion \(\comp R\) need not be. But the category \(\Ring_{\hat\delta}\) is still a 1-category (see Proposition \ref{prop_truncated}).
\end{rem}

\begin{prop} \label{prop_cat}
The category \(\Ring_{\hat\delta}^\an\) is presentable. The forgetful functor \(\Forget:\Ring_{\hat\delta}^\an\to\Ring^\an\) is conservative and admits a right adjoint, denoted  \(\Cofree\). In particular, the forgetful functor creates all colimits.
\end{prop}

\begin{proof}
We can identify \(\Ring_{\hat\delta}^\an\) with the following pullback in \(\widehat{\text{Cat}}_\infty\):
\[\begin{tikzcd}
\Ring_{\hat\delta}^\an \ar[r,"\Forget"] \ar[d] & \Ring^\an \ar[d,"\prod \comp{(-)}"] \\
\prod \comp{(\Ring_{\delta_p}^\an)} \ar[r] & \prod \comp{(\Ring^\an)}.
\end{tikzcd}\]
By \cite[Proposition A.23]{Prismatization} the forgetful functor \(\Ring_{\delta_p}^\an\to\Ring^\an\) admits both adjoints. It follows that \(\comp{(\Ring_{\delta_p}^\an)}\to\comp{(\Ring^\an)}\) has both adjoints as well, because both categories are presentable and the forgetful functor commutes with limits and colimits (c.f.\@ \cite[Corollary 5.5.2.9]{HTT}). By \cite[Proposition 5.5.3.13]{HTT}, limits in \(\PrL\) are computed underlying in \(\widehat{\text{Cat}}_\infty\). Since our pullback diagram consists of left adjoints, it is in fact a pullback in \(\PrL\), so \(\Forget\) is a left adjoint. As conservativeness is clear, the result follows.
\end{proof}

\begin{cor} \label{cor_push}
Let \(\begin{tikzcd}[cramped, sep=small] A & C \ar[l] \ar[r] & B \end{tikzcd}\) be a diagram in \(\Ring_{\hat\delta}^\an\). Then the tensor product \(A\otimes_{C}B\) is naturally an animated \(\hat\delta\)-ring. Moreover
\[\begin{tikzcd}
A \ar[r] & A\otimes_C B \\
C \ar[r] \ar[u] & B \ar[u]
\end{tikzcd}\]
is a pushout in \(\Ring_{\hat\delta}^\an\).
\end{cor}

In fact, we can give a concrete formula for the right adjoint. We denote the \(p\)-typical Witt vectors functor by \(W_p:\Ring^\an\to\Ring^\an_{\delta_p}\). Recall that it defines a right adjoint to the forgetful functor (c.f.\@ appendix to \cite{Prismatization}).

\begin{lem} \label{lem_Witt}
If \(R\) is a \(p\)-complete animated ring, then the \(p\)-typical Witt vectors \(W_p(R)\) are also \(p\)-complete.
\end{lem}

\begin{proof}
We have a commutative square of left adjoints
\[\begin{tikzcd}
\Ring_{\delta_p}^\an \ar[r,"\Forget"] \ar[d,"\comp{(-)}"] & \Ring^\an \ar[d,"\comp{(-)}"] \\
\comp{(\Ring_{\delta_p}^\an)} \ar[r,"\Forget"] & \comp{(\Ring^\an)}.
\end{tikzcd}\]
Passing to right adjoints yields the result.
\end{proof}

\begin{prop} \label{prop_cofree}
Let \(R\) be an animated ring. Then \(\Cofree(R)\) is given as the pullback
\[\begin{tikzcd}
\Cofree(R) \ar[r] \ar[d] & \prod W_p(\comp R) \ar[d] \\ R \ar[r] & \prod \comp R
\end{tikzcd}\]
in \(\Ring^\an\). The \(\hat\delta\)-structure is given, on a prime \(p\), by the canonical \(\delta_p\)-structure on \(W_p(\comp R)\simeq \comp{\Cofree(R)}\).
\end{prop}

\begin{proof}
Let \(\Cofree'(R) := R\times_{\prod \comp R}\prod W_p(\comp R)\) denote this pullback. For any \(A\in\Ring_{\hat\delta}^\an\) we have natural equivalences
\[\Map_{\hat\delta}(A,\Cofree'(R))\simeq \Map(A,R)\times_{\prod\Map(A,\comp R)}\prod\Map_{\delta_p}(A,W_p(\comp R))\]
\[\simeq \Map(A,R)\times_{\prod\Map(A,\comp R)}\prod\Map(A,\comp R)\simeq \Map(A,R).\]
This shows the adjunction relation, and so \(\Cofree(R)\simeq\Cofree'(R)\).
\end{proof}

\begin{cor} \label{cor_cofree}
The functor \(\Cofree\) commutes with geometric realizations and \(\aleph_1\)-filtered colimits.
\end{cor}

\begin{proof}
Geometric realizations and \(\aleph_1\)-filtered colimits commute with countable limits in \(\D(\Z)\). Indeed, they commute with countable products and with finite limits (by stability). Moreover, they preserve connectivity, and so in \(\Ring^\an\) they are computed underlying in \(\D(\Z)\). This means that they commute with \(p\)-completions and Witt vectors. Thus, the description from Proposition \ref{prop_cofree} yields the result. We also use here that the pullback in question is actually a pullback in \(\D(\Z)\), as on \(\pi_0\) the map \(\prod W_p(\comp R)\to \prod\comp R\) is surjective, hence the pullback itself commutes with sifted colimits (by stability).
\end{proof}

We will now explain the behavior of limits and truncations in the category \(\Ring^\an_{\hat\delta}\). The following negative assertion will follow from our later results on truncations.

\begin{prop}
The forgetful functor \(\Ring_{\hat\delta}^\an\to \Ring^\an\) has no left adjoint. In particular, limits of \(\hat\delta\)-rings are not computed underlying.
\end{prop}

\begin{proof}
Since right adjoints preserve truncated objects, the claim follows from Proposition \ref{prop_truncated}.
\end{proof}

\begin{rem} \label{rem_lim}
We can say more explicitly what can go wrong with limits of \(\hat\delta\)-rings, but let us first say how this can be fixed. The problem stems from the fact that limits do not preserve connectivity. If we allow for negative homotopy groups, that is extend our formalism to derived rings of \cite{raksit}, then limits of derived \(\hat\delta\)-rings would be computed underlying. We do not pursue this approach, but there should not be any obstacles in extending our results to this setting.

With this in mind, the problem with limits of animated rings reduces to understanding why connective covers of derived \(\hat\delta\)-rings need not preserve \(\hat\delta\)-structures. This is because \(\pi_1\comp{(\pi_{-1}R)}\) might give a non-zero contribution to \(\pi_0\comp R\). If in addition, the Frobenius lift \(\phi\) sends \(\pi_0R\) non-trivially to this contribution, then the connective cover cannot inherit a \(\hat\delta\)-structure.
\end{rem}

\begin{prop} \label{prop_disclim}
Let \(\{R_i\}\) be a diagram of animated \(\hat\delta\)-rings whose limit in \(\D(\Z)\) is connective. Then the underlying animated ring of the limit \(\lim\, R_i\) in \(\Ring_{\hat\delta}^\an\) agrees with the limit \(\lim\, R_i\) in \(\Ring^\an\). In other words, the forgetful functor \(\Ring^\an_{\hat\delta}\to\Ring^\an\) creates all limits which can be computed underlying in \(\D(\Z)\).
\end{prop}

\begin{proof}
Let \(R=\lim\, R_i\) in \(\Ring^\an\). By assumption, it can be computed underlying in \(\D(\Z)\) where both the \(p\)-completion and the mod \(p\) reduction commute with limits. Hence, we get \(\comp R\simeq \lim\, \comp{(R_i)}\) and \(R/p\simeq \lim\, R_i/p\) in \(\D(\Z)\). Since \(\comp R\) and \(R/p\) are connective, those identities are also valid in \(\Ring^\an\).

Therefore, taking the limit of the diagrams
\[\begin{tikzcd}
& \comp{(R_i)} \ar[d,"\can"] \\ R_i \ar[ur,"{\phi_i}"] \ar[r,"\Frob"] & R_i/p
\end{tikzcd}\]
induces a \(\hat\delta\)-structure on \(R\). It is then clear from the construction that \(R\) satisfies the universal property for the limit of \(R_i\) in \(\Ring_{\hat\delta}^\an\).
\end{proof}

\begin{cor}
The forgetful functor \(\Ring_{\hat\delta}^\an\to\Ring^\an\) creates arbitrary products.
\end{cor}

\begin{cor}
Let \(R\) be an animated \(\hat\delta\)-ring. Then the fracture square
\[\begin{tikzcd}
R \ar[r] \ar[d] & \prod \comp R\ar[d] \\ R_\Q \ar[r] & \big(\prod \comp R\big)_\Q
\end{tikzcd}\]
is a pullback of animated \(\hat\delta\)-rings.
\end{cor}

\begin{cor}
The forgetful functor \(\Ring_{\hat\delta}^\an\to \Ring^\an\) is comonadic.
\end{cor}

\begin{proof}
Note that the inclusion \(\D(\Z)_{\geq 0}\to\D(\Z)\) is comonadic, because the counit axiom yields a splitting to \(\tau_{\geq 0}A \to A\). The dual to the Barr-Beck-Lurie theorem then implies that split totalizations preserve connectivity. Hence, by Proposition \ref{prop_disclim}, the forgetful functor \(\Ring_{\hat\delta}^\an\to \Ring^\an\) creates split totalizations, and so is comonadic, again by the dual to the Barr-Beck-Lurie theorem.
\end{proof}

We will now investigate truncations. We begin with the following computation.

\begin{lem} \label{lem_truncofree}
Let \(R\) be an \(n\)-truncated animated ring.
\begin{enumerate}
\item If there is \(p\) for which \(\pi_{n+1}\comp R \ne 0\), then \(\Cofree(R)\) is an \((n+1)\)-truncated animated ring which is not \(n\)-truncated.
\item If not, then \(\Cofree(R)\) is an \(n\)-truncated animated ring.
\end{enumerate}
\end{lem}

\begin{proof}
We have \(\Omega^\infty W_p(\comp R)\simeq \prod_\N \Omega^\infty\comp R\). Thus, \(\pi_{n+1}W_p(\comp R)\simeq \prod_\N\pi_{n+1}\comp R\) and the result follows from the formula of Proposition \ref{prop_cofree}.
\end{proof}

\begin{lem} \label{lem_trunc}
Let \(R\) be an animated \(\hat\delta\)-ring for which \(R/p\), hence also \(\comp R\), is \((n+1)\)-truncated as an animated ring. Then, for every animated \(\delta_p\)-ring \(A\), the fibers of the forgetful map \(\Map_{\delta_p}(A,\comp R)\to \Map(A,\comp R)\) are \(n\)-truncated.
\end{lem}

\begin{proof}
We have a pullback square
\[\begin{tikzcd}
\Map_{\delta_p}(A,\comp R) \ar[r] \ar[d] & \Map(A, \comp R) \ar[d] \\
\Map\Big(\resizebox{1.7cm}{!}{\begin{tikzpicture} \node(1) at (-0.5,-0.5) {\(A\)}; \node(2) at (1,-0.5) {\(A/p\)}; \node(3) at (1,1) {\(A\)}; \draw[->] (1) edge node[above=1,node font=\tiny]{\(\Frob\)} (2); \draw[->] (1) edge node[above left,node font=\tiny]{\(\phi_A\)} (3); \draw[->] (3) edge node[right,node font=\tiny]{\(\can\)} (2); \end{tikzpicture}},\resizebox{1.7cm}{!}{\begin{tikzpicture} \node(1) at (-0.5,-0.5) {\(\comp R\)}; \node(2) at (1,-0.5) {\(R/p\)}; \node(3) at (1,1) {\(\comp R\)}; \draw[->] (1) edge node[above=1,node font=\tiny]{\(\Frob\)} (2); \draw[->] (1) edge node[above left,node font=\tiny]{\(\phi_R\)} (3); \draw[->] (3) edge node[right,node font=\tiny]{\(\can\)} (2); \end{tikzpicture}}\Big) \ar[r] & \Map\Big(\resizebox{1.7cm}{!}{\begin{tikzpicture} \node(1) at (-0.5,-0.5) {\(A\)}; \node(2) at (1,-0.5) {\(A/p\)}; \node(3) at (1,1) {\(A\)}; \draw[->] (1) edge node[above=1,node font=\tiny]{\(\Frob\)} (2); \draw[->] (3) edge node[right,node font=\tiny]{\(\can\)} (2); \end{tikzpicture}},\resizebox{1.7cm}{!}{\begin{tikzpicture} \node(1) at (-0.5,-0.5) {\(\comp R\)}; \node(2) at (1,-0.5) {\(R/p\)}; \node(3) at (1,1) {\(\comp R\)}; \draw[->] (1) edge node[above=1,node font=\tiny]{\(\Frob\)} (2); \draw[->] (3) edge node[right,node font=\tiny]{\(\can\)} (2); \end{tikzpicture}}\Big).
\end{tikzcd}\]
Hence, it suffices to show that the bottom horizontal map has \(n\)-truncated fibers. Over a point given by
\[\begin{tikzcd}
& A\ar[d,"\can"] \ar[rrdd,"f_1"] \\ A \ar[rrdd,"f_0"] \ar[r,"\Frob"] & A/p \ar[rrdd,"f_2"] \\[-40] & &[-45] & \comp R \ar[d,"\can"] \\ & & \comp R \ar[r,"\Frob"] & R/p
\end{tikzcd}\]
its fiber can be identified with the fiber of
\[\begin{tikzcd}
\{\phi_Rf_0\}\times_{\Map(A,\comp R)} \{f_1\phi_A\} \ar[r] & \Omega_{\Frob f_0}\Map(A,R/p)
\end{tikzcd}\]
over the trivial loop. As both \(\comp R\) and \(R/p\) are \((n+1)\)-truncated, those spaces are \(n\)-truncated, so the fiber as well.
\end{proof}

As with limits (c.f.\@ Remark \ref{rem_lim}), truncations are not underlying because \(p\)-completion intertwines neighbouring homotopy groups. Nonetheless, we can give bounds on the discrepancy.

\begin{prop} \label{prop_truncated}
Let \(R\) be an animated \(\hat\delta\)-ring. For any \(n\) the following hold.
\begin{enumerate}
\item If \(R\) is an \(n\)-truncated object of \(\Ring^\an_{\hat\delta}\), then \(R\) is an \((n+1)\)-truncated animated ring.
\item If \(R\) is an \(n\)-truncated animated ring, then it is an \(n\)-truncated object in \(\Ring^\an_{\hat\delta}\).
\item There exist \(n\)-truncated objects in \(\Ring^\an_{\hat\delta}\) which are not \(n\)-truncated as animated rings. In particular, truncations in \(\Ring_{\hat\delta}^\an\) are not computed underlying.
\end{enumerate}
\end{prop}

\begin{proof}
We first show (1). Let \(\Free:\D(\Z)_{\geq 0}\to\Ring^\an\) and \(\Freep:\D(\Z)_{\geq 0}\to \Ring^\an_{\delta_p}\) denote left adjoints to the forgetful functors. On a projective \(\Z^{\oplus I}\) the functor \(\Free\) is given by \(\Z[x_i]_I\) and \(\Freep\) by \(\Z\{x_i\}_I\), the free \(\delta_p\)-ring on \(I\) generators. Consider the pullback
\[\begin{tikzcd}
A \ar[r] \ar[d] & \prod \comp{\Freep(\Z[1])} \ar[d] \\
\Free(\Z[1])_\Q \ar[r] & \big(\prod \comp{\Freep(\Z[1])}\big)_\Q
\end{tikzcd}\]
in \(\D(\Z)\). Since \(\Z[1]\) is equivalent to \(*\sqcup_{\Z} *\) in \(\D(\Z)_{\geq 0}\), we get that \(\pi_0\Free(\Z[1])\simeq\pi_0\Freep(\Z[1])\simeq\Z\). Thus, \(A\) is connective, and so this is a pullback square in \(\Ring^\an\). In fact, this is automatically the fracture square for \(A\), and so \(A\) inherits a \(\hat\delta\)-structure from the canonical \(\delta_p\)-structures on \(\comp{\Freep(\Z[1])}\).

Consider the mapping space \(\Map_{\hat\delta}(A,R)\). By assumption, it is \(n\)-truncated. On the other hand, by fracturing, we have a pullback square
\[\begin{tikzcd}
\Map_{\hat\delta}(A,R) \ar[r] \ar[d] & \prod\Map_{\delta_p}\big(\comp{\Freep(\Z[1])},\comp R\big) \ar[d] \\ \Map\big(\Free(\Z[1])_\Q,R_\Q\big) \ar[r] & \Map\Big({\Free(\Z[1])}_\Q,(\prod \comp R)_\Q\Big).
\end{tikzcd}\]
By adjunction, this is equivalent to
\[\begin{tikzcd}
\Map_{\hat\delta}(A,R) \ar[r] \ar[d] & \prod\Map(\comp{\Z[1]},\comp R) \ar[d] \\ \Map\big(\Z[1]_\Q,R_\Q\big) \ar[r] & \Map\Big({\Z[1]}_\Q,\big(\prod \comp R\big)_\Q\Big),
\end{tikzcd}\]
which, again by fracturing, yields \(\Map_{\hat\delta}(A,R)\simeq \Map(\Z[1],R)\simeq\tau_{\geq 0}(R[-1])\), and so \(R\) is \((n+1)\)-truncated, as desired.

We will now show (2). For any \(B\in\Ring^\an_{\hat\delta}\) we have
\[\Map_{\hat\delta}(B,R)\simeq \Map(B,R)\times_{\prod \Map(\comp B, \comp R)}\prod\Map_{\delta_p}(\comp B,\comp R).\]
Since \(R\) is \(n\)-truncated, the mapping space \(\Map(B,R)\) is \(n\)-truncated and the other two mapping spaces are \((n+1)\)-truncated. Thus, it suffices to show that the fibers of \(\Map_{\delta_p}(\comp B,\comp R)\to \Map(\comp B,\comp R)\) are \(n\)-truncated. This follows from Lemma \ref{lem_trunc}.

Assertion (3) follows from Lemma \ref{lem_truncofree} and the fact that right adjoints preserve truncated objects. More explicitly, take any \(n\)-truncated animated ring \(R\) for which \(\pi_{n+1}\comp R \ne 0\) for some \(p\). Then \(\Cofree(R)\) is an \(n\)-truncated object of \(\Ring_{\hat\delta}^\an\) which is not an \(n\)-truncated animated ring.
\end{proof}

That truncations are not underlying stems form the fact that the Frobenius lift \(\phi:R\to\comp R\) can induce a non-zero map \(\pi_{n+1}R\to \pi_{n+1}\comp R\to \pi_1\comp{(\pi_nR)}\). If that happens, there can be no \(\hat\delta\)-structure on \(\tau_{\leq n}R\) compatible with the truncation map \(R\to \tau_{\leq n}R\).

\begin{prop} \label{prop_nortrun}
Let \(R\) be an animated \(\hat\delta\)-ring. Assume \(\pi_{n+1}R\to \prod\pi_1\comp{(\pi_nR)}\) induced by Frobenius lifts is zero. Then \(\tau_{\leq n}R\) admits a unique \(\hat\delta\)-structure such that the canonical map \(R\to\tau_{\leq n}R\) can be refined to a \(\hat\delta\)-morphism.
\end{prop}

\begin{proof}
By Proposition \ref{prop_trunpdelta} applied to \(\comp R\), we have a commutative diagram
\[\begin{tikzcd}
& \tau_{\leq n} \comp R \ar[d,"\can"] \\[8]
\tau_{\leq n}\comp R \ar[ur,"\tau_{\leq n}\phi"] \ar[r,"\Frob"] & (\tau_{\leq n}R)/p.
\end{tikzcd}\]
As \(\comp{(\tau_{\leq n} R)}\) is \((n+1)\)-truncated, we also have a commuting square
\[\begin{tikzcd}
\tau_{\leq n+1}R \ar[r,"\tau_{\leq n+1}\phi"] \ar[d] &[10] \comp{(\tau_{\leq n} R)} \ar[d,"\can"] \\
\tau_{\leq n}\comp R \ar[r,"\tau_{\leq n}\phi"] & \tau_{\leq n}\comp R.
\end{tikzcd}\]
Since the top horizontal map induces zero on \(\pi_{n+1}\) by assumption, we get that \(\tau_{\leq n}\phi\) factors uniquely through \(\comp{(\tau_{\leq n}R)}\) endowing it with a \(\delta_p\)-structure. This argument also makes \(\comp R \to \comp{(\tau_{\leq n}R)}\) into a \(\delta_p\)-morphism. Repeating this for every prime yields the result.
\end{proof}

It is moreover true that, under the hypothesis of Proposition \ref{prop_nortrun}, the \(\hat\delta\)-map \(R\to \tau_{\leq n}R\) agrees with \(n\)-truncation in \(\Ring_{\hat\delta}^\an\). This follows from the following characterization of truncated \(\hat\delta\)-rings.

\begin{prop} \label{prop_truncond}
Let \(R\) be an animated \(\hat\delta\)-ring. Then \(R\) is an \(n\)-truncated object in \(\Ring_{\hat\delta}^\an\) if and only if the following two conditions hold:
\begin{enumerate}
\item \(R\) is an \((n+1)\)-truncated animated ring,
\item there is no non-zero \(M\subset \pi_{n+1}R\) for which \(\phi_R(M)\subset \text{\normalfont Im}(\pi_0\comp M \to \pi_{n+1}\comp R)\) for all \(p\).
\end{enumerate}
\end{prop}

\begin{proof}
First let us assume conditions (1) and (2). We want to show that \(R\) is \(n\)-truncated as an animated \(\hat\delta\)-ring.

Note that, taking \(M\subset \pi_{n+1}R\) to be the submodule of arbitrarily \(p\)-divisible elements, condition (2) implies \(M=0\), because such elements vanish in \(p\)-complete groups and are preserved under \(\phi_R\). Similarly, if \(x\in \pi_{n+1}\comp R\) satisfies \(p^nx=0\), then, by the cofiber sequence \(\comp R\xrightarrow{p^n}\comp R \to R/p^n\), it comes from the image of \((\pi_{n+1}R)[p^n]\), so, taking \(M\) to be the submodule of \(p\)-power torsion elements, condition (2) also implies \(M=0\). Hence, \(R/p\) is \((n+1)\)-truncated and the \(p\)-completion \(\pi_{n+1}R\to\comp{(\pi_{n+1}R)}\) is injective.

Let \(A\) be any animated \(\hat\delta\)-ring. By Proposition \ref{prop_truncated} (2) we know that \(\Map_{\hat\delta}(A,R)\) is \((n+1)\)-truncated. Hence, we just need to show that \(\pi_{n+1}\Map_{\hat\delta}(A,R)=0\). Looking at the long exact sequence associated to the pullback
\[\Map_{\hat\delta}(A,R)\simeq \Map(A,R)\times_{\prod \Map(\comp A, \comp R)}\prod\Map_{\delta_p}(\comp A,\comp R)\]
and using that \(\comp R\) is \((n+1)\)-truncated we learn that \(\pi_{n+1}\Map_{\hat\delta}(A,R)\) is isomorphic to the kernel of
\[\pi_{n+1}\Map(A,R)\times \prod\pi_{n+1}\Map_{\delta_p}(\comp A,\comp R) \to \prod\pi_{n+1}\Map(\comp A,\comp R).\]
By Lemma \ref{lem_trunc}, the map \(\prod\pi_{n+1}\Map_{\delta_p}(\comp A,\comp R) \to \prod\pi_{n+1}\Map(\comp A,\comp R)\) is injective. Therefore, it remains to show that \(\pi_{n+1}\Map(A,R) \to \prod\pi_{n+1}\Map(\comp A,\comp R)\) is injective and that those subsets have zero intersection.

We first show injectivity. As \(R\) is \((n+1)\)-truncated, we get that \(\pi_{n+1}\Map(A,R)\) is a subset of \(\Hom(\pi_0A,\pi_{n+1}R)\). Similarly, \(\pi_{n+1}\Map(\comp A,\comp R)\subset \Hom(\pi_0\comp A,\pi_{n+1}\comp R)\) and image of \(p\)-completion lies in \(\Hom(\comp{(\pi_0A)},\comp{(\pi_{n+1}R)})\subset \Hom(\pi_0\comp A,\pi_{n+1}\comp R)\). Since \(p\)-completion \(\Hom(\pi_0A,\pi_{n+1}R)\to\Hom(\comp{(\pi_0A)},\comp{(\pi_{n+1}R)})\) is injective, we get what we wanted.

It remains to show that the intersection is zero. A morphism \(f:\pi_0 A \to \pi_{n+1} R\) that after \(p\)-completion lies in the image of \(\pi_{n+1}\Map_{\delta_p}(\comp A, \comp R)\) must satisfy \(\phi_R f=\comp f\phi_A\). But in that case \(M=\text{Im}(f)\) satisfies \(\phi_R(M)\subset \text{\normalfont Im}(\pi_0\comp M \to \pi_{n+1}\comp R)\), so condition (2) implies \(f=0\), as required.

To finish the proof, we will show that an \(n\)-truncated object \(S\) of \(\Ring_{\hat\delta}^\an\) satisfies condition (2). Condition (1) follows from Proposition \ref{prop_truncated} (1).

Assume to the contrary that there is a non-zero \(M\subset \pi_{n+1}S\) such that \(\phi_S(M)\subset \text{\normalfont Im}(\pi_0\comp M \to \pi_{n+1}\comp S)\) for all \(p\). Let us choose \(\{x_i\}_I\) abelian group generators of \(M\). Note that they are also \(p\)-adic generators for \(\pi_0\comp M\), because the fiber of \(\comp{(\Z^{\oplus I})}\to \comp M\) is connective. Hence, for every \(p\), we can find a lift \(\Z^{\oplus I}\to \comp{(\Z^{\oplus I})}\) of \(\phi_S|_M\) which is moreover divisible by \(p\) (because \(\phi_S\) vanishes mod \(p\) on higher homotopy groups by \cite{BSWitt}[Proposition 11.6]). They induce a \(\hat\delta\)-structure on \(\bigotimes_p\Z_{(p)}[x_i]/(x_i^p)\). By construction, the canonical element of \(\pi_{n+1}\Map\left(\bigotimes_p\Z_{(p)}[x_i]/(x_i^p),S\right)\) given by the composite \(\Z^{\oplus I} \to M \subset \pi_{n+1}S\) lifts to \(\pi_{n+1}\Map_{\hat\delta}\left(\bigotimes_p\Z_{(p)}[x_i]/(x_i^p),S\right)\). But then \(S\) is not \(n\)-truncated, contradiction.
\end{proof}

We can also use a stronger truncation procedure which both preserves \(\hat\delta\)-structures and is more computable then the categorical one. We describe it in Construction \ref{con_tauhat}. But before we do so, we introduce the following terminology.

\begin{defi}
Let \(M\) be an abelian group. We say that \(0\in M\) is \(p^\infty\)-\emph{divisible} if there exist elements \(x_0,x_1,...\in M\) satisfying \(px_{i+1}=x_i\), \(px_0=0\), and \(x_k\ne 0\) for some \(k\).
\end{defi}

\begin{ex}
If \(M\) has bounded \(p^\infty\)-torsion, then \(0\in M\) is not \(p^\infty\)-divisible.
\end{ex}

\begin{lem} \label{lem_pi1}
Let \(M\) be an abelian group. Then
\[\pi_1\comp M=\{(x_0,x_1,...)\in M^\N \,|\, px_0=0\ \text{and}\ px_{i+1}=x_i\}.\]
\end{lem}

\begin{proof}
We have \(\comp M=\lim\, M/p^n\). Hence, \(\pi_1\comp M\) can be identified with the 1-categorical limit of the inverse system \(...\to M[p^2]\to M[p]\to 0\) with transition maps given by multiplication by \(p\). The formula follows.
\end{proof}

\begin{cor} \label{cor_pi1tor}
Let \(M\) be an abelian group. Then \(\pi_1\comp M\) is torsion free and has no non-zero infinitely \(p\)-divisible elements.
\end{cor}

\begin{lem}
Let \(X\) be an \(n\)-truncated spectrum. Then
\[\pi_{n+1}\comp X\simeq \pi_1\comp{(\pi_n X)}.\]
\end{lem}

\begin{proof}
Follows from the fiber sequence \(\comp{(\pi_nX)}\to\comp X\to \comp{(\tau_{<n}X)}\).
\end{proof}

\begin{cor} \label{cor_pntrun}
Let \(X\) be an \(n\)-truncated spectrum. The following conditions are equivalent:
\begin{enumerate}
\item \(\comp X\) is \(n\)-truncated.
\item \(0\in\pi_nX\) is not \(p^\infty\)-divisible.
\end{enumerate}
\end{cor}

\begin{con} \label{con_tauhat}
Let \(R\) be an animated ring. We denote by \(\hat\tau_{\leq n}R\) the pullback of the diagram
\[\begin{tikzcd}
\hat\tau_{\leq n}R \ar[r] \ar[d] & \prod\tau_{\leq n}(\comp{R}) \ar[d] \\
(\tau_{\leq n}R)_\mathbb Q \ar[r] & \big(\prod\tau_{\leq n}(\comp{R})\big)_\mathbb Q.
\end{tikzcd}\]
Note that the right vertical and the bottom horizontal maps are jointly surjective on \(\pi_0\), so \(\hat\tau_{\leq n}R\) is connective and this pullback can be computed underlying in \(\D(\Z)\). Moreover, this is already the fracture square for \(\hat\tau_{\leq n}R\). In particular, the \(p\)-completion of \(\hat\tau_{\leq n}R\) is \(n\)-truncated. In the case \(n=0\) we also write \(\hat\pi_0\) for \(\hat\tau_{\leq 0}\).
\end{con}

Another motivation for introducing \(\hat\pi_0\) is to make use of the set-theoretical definition of \(\delta_p\)-rings, which only works in the discrete setting.

\begin{rem}
We have a fiber sequence
\[\begin{tikzcd}[sep=15]
\hat\pi_0(\pi_nR)[n] \ar[r] & \hat\tau_{\leq n}R \ar[r] & \tau_{\leq n-1}R.
\end{tikzcd}\]
\end{rem}

\begin{rem}
There is a comparison map \(\tau_{\leq n} R\to \hat\tau_{\leq n}R\), which is a rational isomorphism. After \(p\)-completion it agrees with the \(n\)-truncation of \(\comp{(\tau_{\leq n}R)}\). In particular, by Corollary \ref{cor_pntrun}, if \(0\in\pi_nR\) is not \(p^\infty\)-divisible for any \(p\), then \(\hat\tau_{\leq n}R\simeq\tau_{\leq n}R\).
\end{rem}

\begin{prop} \label{prop_tauhat}
Let \(R\) be an animated \(\hat\delta\)-ring. Then \(\hat\tau_{\leq n}R\) is naturally a \(\hat\delta\)-ring and the canonical map \(R\to \hat\tau_{\leq n}R\) admits a refinement to a \(\hat\delta\)-morphism. Moreover, \(\hat\tau_{\leq n}R\) is an \(n\)-truncated object of \(\Ring^\an_{\hat\delta}\).
\end{prop}

\begin{proof}
The first statement follows from the identity \(\comp{(\hat\tau_{\leq n}R)}\simeq \tau_{\leq n}\comp R\) and the fact that truncations of \(\delta_p\)-rings are computed underlying. Second statement follows from Proposition \ref{prop_truncated}.
\end{proof}

\begin{rem}
By \ref{prop_disclim}, for any \(R\in\Ring_{\hat\delta}^\an\), we have
\[R\simeq \lim\,\hat\tau_{\leq n}R\]
in the category of animated \(\hat\delta\)-rings.
\end{rem}

For completeness, we also provide a more algebraic description of \(\hat\pi_0\).

\begin{prop}
The functor \(\hat\pi_0:\Ab\to\Ab\) is a localisation with essential image the full subcategory of those abelian groups that admit no non-zero map from \(\Q/\Z\).
\end{prop}

\begin{proof}
Follows from identities \((\Q/\Z)_\Q\simeq 0\) and \(\comp{(\Q/\Z)}\simeq\Z_p[1]\).
\end{proof}

\begin{rem}
Let \(R\in\Ring\). Comparing fracture squares for \(R\) and \(\hat\pi_0R\) we get the formula \(\hat\pi_0R=R/I\), where
\[I=\left(\prod\pi_1\comp R\right)_\Q\Big /\left(\prod\pi_1\comp R\right)\subset R.\]
\end{rem}

\subsection{\(\hat\delta\)-schemes} \label{sec_hatschemes}

In this section, we finally define \(\hat\delta\)-schemes. We further prove certain structural properties, which imply that we can interact with \(\hat\delta\)-schemes as with normal schemes. They admit a basis by affines, maps of \(\hat\delta\)-schemes glue, finite limits are computed via the tensor product of rings, and so forth.

\begin{con} \label{con_loc}
Let \(R\) be an animated \(\hat\delta\)-ring and let \(f\in \pi_0R\). We want to refine the canonical map \(R\to R[f^{-1}]\) to a morphism of \(\hat\delta\)-rings.

Let \(\phi : R \to \comp R\) be the Frobenius lift for \(R\). Since on the mod \(p\) reduction \(\phi\) is the Frobenius and \(f^p\) is invertible in \(R[f^{-1}]/p\), we get that \(\phi(f)\) is invertible in \(\comp{(R[f^{-1}])}\). Therefore, we get an induced map \(\phi':R[f^{-1}]\to\comp{(R[f^{-1}])}\) which constitutes a Frobenius lift on \(R[f^{-1}]\) compatible with the one on \(R\).
\end{con}

\begin{rem} \label{rem_loc}
Let \(R\) be an animated \(\hat\delta\)-ring. It follows from Construction \ref{con_loc} that precomposition with \(R\to R[f^{-1}]\) induces a fully faithful map \(\Map_{\hat\delta}(R[f^{-1}],S)\) \(\to \Map_{\hat\delta}(R,S)\) with essential image those \(\hat\delta\)-ring maps that send \(f\) to an invertible element of \(\pi_0S\). In particular, the space of refinements of \(R\to R[f^{-1}]\) to a \(\hat\delta\)-map admits an initial object, and so is contractible.
\end{rem}

\begin{war}
Note that if \(R\) is a non \(p\)-complete \(\delta_p\)-ring then the localization at \(f\in\pi_0R\) need not carry an induced \(\delta_p\)-structure. On the other hand, even if \(R\) is \(p\)-complete, \(R[f^{-1}]\) might not be.
\end{war}

Let \(X\) be a derived scheme (i.e. modelled on animated rings). We can consider the small site \(X_{\text{Zar}}^{\text{dis}}\) whose objects are affine opens of \(X\) and whose morphisms are the distinguished embeddings (those given by the non-vanishing locus of a single global section). Construction \ref{con_loc} induces a functor \(\Str_{\hat\delta}:(X_{\text{Zar}}^{\text{dis}})^\op\to \An\) sending an affine open \(\Spec A\) to the space of \(\hat\delta\)-structures on \(A\).

\begin{prop} \label{prop_glue}
The functor \(\Str_{\hat\delta}:(X_{\Zar}^\dis)^\op\to \An\) is a sheaf.
\end{prop}

\begin{proof}
Follows from the fact that \(\O_X:(X_{\Zar}^\dis)^\op\to \Ring^\an\) is a sheaf and from Proposition \ref{prop_disclim}.
\end{proof}

We can also consider the site \(X_\Zar^\aff\) of affine opens and all open embeddings. Since the topology on \(X_\Zar^\aff\) is generated by \(X_\Zar^\dis\) we get, by the process of right Kan extension:

\begin{cor} \label{cor_glue}
The functor \(\Str_{\hat\delta}:(X_\Zar^\dis)^\op\to \An\) extends to a functor \(\Str_{\hat\delta}:(X_{\Zar}^\aff)^\op\to \An\) which is again a sheaf.
\end{cor}

\begin{rem}
\(\hat\delta\)-rings also glue in the étale topology, but we do not go into this generality (see \cite{prismatic}).
\end{rem}

Suppose we are given a sheaf of animated \(\hat\delta\)-rings \((X_\Zar^\aff)^\op\to\Ring_{\hat\delta}^\an\). Since the category \(\Ring_{\hat\delta}^\an\) admits all limits we can formally extend any such \(\hat\delta\)-valued sheaf to a sheaf \(\mathcal F:\Open(X)^\op\to\Ring_{\hat\delta}^\an\) (c.f.\@ \cite[Proposition 1.1.4.4]{SAG}).

It is generally not true that the composite \(\Forget\circ\mathcal F:\Open(X)^\op\to\Ring_{\hat\delta}^\an\to \Ring^\an\) remains a sheaf, because \(\Forget\) does not commute with limits. We will refer to the sheafification of \(\Forget\circ\mathcal F\) as the \emph{underlying sheaf of animated rings} of \(\mathcal F\).

We are now ready to define \(\hat\delta\)-schemes.

\begin{con} \label{con_aff}
Let \(R\) be an animated \(\hat\delta\)-ring and let \(X\) be the topological space \(\Spec R\). We define a sheaf of \(\hat\delta\)-rings \(\O_X^{\hat\delta}:(X_\Zar^\dis)^\op\to \Ring^\an_{\hat\delta}\) by the obvious assignment \(D(f)\mapsto R[f^{-1}]\). By right Kan extension, it induces a sheaf on all of \(\Spec R\).

Invoking Proposition \ref{prop_disclim} we get that the composite \(\Forget\circ\O_X^{\hat\delta}\) agrees with the structure sheaf of \(\Spec R\) on affine opens. This implies that the underlying sheaf of animated rings of \(\O_X^{\hat\delta}\) agrees with the usual structure sheaf on \(\Spec R\). We refer to the pair \((\Spec R,\O_{\Spec R}^{\hat\delta})\) as an \emph{affine \(\hat\delta\)-scheme}.
\end{con}

\begin{defi}
A \emph{\(\hat\delta\)-scheme} is a topological space \(X\) together with a sheaf of \(\hat\delta\)-rings \(\O_X^{\hat\delta}:\Open(X)^\op\to\Ring^\an_{\hat\delta}\) such that \(X\) admits a covering \(\{U_i\}\) for which \((U_i,\O_X^{\hat\delta}|_{U_i})\) is equivalent to an affine \(\hat\delta\)-scheme. We denote by \(\O_X\) the underlying sheaf of animated rings of \(\O_X^{\hat\delta}\).
\end{defi}

\begin{prop}
Let \(X\) be a \(\hat\delta\)-scheme. Then the pair \((X,\O_X)\) forms a derived scheme.
\end{prop}

\begin{proof}
It suffices to show the result for an affine \(\hat\delta\)-scheme (c.f.\@ \cite[Corollary 1.1.6.4]{SAG}). This we have already done in Construction \ref{con_aff}.
\end{proof}

\begin{rem}
Let \(X\) be a \(\hat\delta\)-scheme and \(U\subset X\) a non-affine open subset. Then \(\O_X^{\hat\delta}(U)\) might not agree with the ring \(\O_X(U)\). As mentioned in Remark \ref{rem_lim}, this problem would disappear when working with derived rings.
\end{rem}

\begin{prop} \label{prop_dschstr}
Let \(X\) be a derived scheme and let \(\O_X^{\hat\delta}\) be a sheaf of animated \(\hat\delta\)-rings on \(X\). The following conditions are equivalent:
\begin{enumerate}
\item \((X,\O_X^{\hat\delta})\) is a \(\hat\delta\)-scheme with \(X\) as the underlying derived scheme.
\item the composition \(\Forget\circ \O_X^{\hat\delta}\) agrees with \(\O_X\) when restricted to \(X_\Zar^\dis\).
\end{enumerate}
\end{prop}

\begin{proof}
(1)\(\Rightarrow\)(2) follows from Construction \ref{con_aff} and by definition. We prove (2)\(\Rightarrow\)(1). Let \(\{U_i\}\) be an affine open cover of \(X\). It suffices to show that each \((U_i,\O_X^{\hat\delta}|_{U_i})\) is an affine \(\hat\delta\)-scheme. Since, by assumption, \(\O_X(U_i)\) agrees with \(\O_X^{\hat\delta}(U_i)\), we have a \(\hat\delta\)-structure on \(\O_X(U_i)\). Hence, it is enough to show that \((U_i,\O_X^{\hat\delta}|_{U_i})\simeq \Spec\O_X^{\hat\delta}(U_i)\) as \(\hat\delta\)-ringed spaces. The underlying topological space is the same and, by Remark \ref{rem_loc}, their structure sheaves agree on all distinguished opens, hence they are equivalent.
\end{proof}

\begin{prop}
Let \(X\) be a \(\hat\delta\)-scheme. Then \(X\) is an affine \(\hat\delta\)-scheme if and only if the underlying derived scheme of \(X\) is affine.
\end{prop}

\begin{proof}
The (\(\Rightarrow\)) direction follows from Construction \ref{con_aff}. The (\(\Leftarrow\)) direction follows from Proposition \ref{prop_dschstr} and Remark \ref{rem_loc} (as in the proof of Proposition \ref{prop_dschstr}).
\end{proof}

\begin{defi} \label{defi_mor}
A morphism of \(\hat\delta\)-schemes is a continuous map \(f:X\to Y\) and a morphism of \(\hat\delta\)-ring valued sheaves \(\O_Y^{\hat\delta}\to f_*\O_X^{\hat\delta}\) that induces a morphism of derived schemes \(X\to Y\). We denote the category of \(\hat\delta\)-schemes by \(\dSch_{\hat\delta}\).
\end{defi}

\begin{prop} \label{prop_pullund}
The category \(\dSch_{\hat\delta}\) admits finite limits. Moreover, they are created by the forgetful functor \(\dSch_{\hat\delta}\to\dSch\).
\end{prop}

\begin{proof}
Follows from Corollary \ref{cor_push}.
\end{proof}

We say that a morphism of \(\hat\delta\)-schemes \(i:U\to X\) is an \emph{open immersion} if the underlying map of schemes is an open immersion and the map \(\O_X^{\hat\delta}\to i_*\O_U^{\hat\delta}\) is an equivalence when restricted to \(U\).

\begin{con}
We put a Grothendieck topology on \(\dSch_{\hat\delta}\) whose covers are generated by collections \(\{U_i\to X\}\) such that all \(U_i\to X\) are open immersions and the induced map of topological spaces \(\bigsqcup U_i\to X\) is surjective. By construction, the forgetful functor \(\dSch_{\hat\delta}\to\dSch\) sends covers to covers.
\end{con}

\begin{prop} \label{prop_gluemaps}
The Grothendieck topology on \(\dSch_{\hat\delta}\) is subcanonical. That is, \(\Map_{\dSch_{\hat\delta}}(-,X):\dSch_{\hat\delta}^\op\to\An\) is a sheaf for any \(\hat\delta\)-scheme \(X\).
\end{prop}

\begin{proof}
In view of Proposition \ref{prop_dschstr}, a morphism of \(\hat\delta\)-schemes \(f:Y\to X\) consists of a morphism of underlying derived schemes together with a refinement of \(\O_X^{\hat\delta}\to f_*\O_Y^{\hat\delta}\) to a morphism of \(\hat\delta\)-ring valued sheaves on \(X_\Zar^\dis\).

The first piece of data is a sheaf on \(\dSch\). It therefore suffices to show that maps of \(\hat\delta\)-ring valued sheaves are local on the domain. This follows from the fact that \(\O_Y^{\hat\delta}\) is a sheaf.
\end{proof}

\begin{prop}
Let \(X\in\dSch_{\hat\delta}\) and \(R\in\Ring^\an_{\hat\delta}\). Taking global sections induces a natural equivalence
\[\Map_{\dSch_{\hat\delta}}(X,\Spec R) \simeq \Map_{\hat\delta}(R,\O_X^{\hat\delta}(X)).\]
In particular, \(\text{\normalfont Spec}(-)\) defines a fully faithful functor \((\Ring^\an_{\hat\delta})^\op\to\dSch_{\hat\delta}\) which is right adjoint to the global sections functor.
\end{prop}

\begin{proof}
If \(X\) is affine, then the result is clear. If \(X\) is separated, then it admits a cover by affine opens whose intersections are still affine. The result then follows from Proposition \ref{prop_gluemaps} and the fact that \(\O^{\hat\delta}_X\) is a sheaf. Lastly, we can cover a general \(X\) with affine opens. As their intersections are separated, we conclude again from Proposition \ref{prop_gluemaps} and the fact that \(\O^{\hat\delta}_X\) is a sheaf.
\end{proof}

\section{Lifts to the sphere} \label{chap_lifts}

Let us formally define what we mean by a lift.

\begin{defi*}
Let \(R\) be a discrete ring. A \emph{lift of \(R\) to \(\S\)} is a connective \(\E_\infty\)-ring \(A\) such that \(A\otimes \Z\simeq R\) as \(\E_\infty\)-rings.
\end{defi*}

This naturally extends to schemes.

\begin{defi*}
Let \(X\) be a discrete scheme. A \emph{lift of \(X\) to \(\Spec\S\)} is a connective spectral scheme \(Y\) such that \(Y\times_{\Spec\S}\Spec\Z\simeq X\).
\end{defi*}

\subsection{Moore schemes} \label{sec_moorescheme}

The theory of Moore spectra (see Appendix \ref{sec_Mtor}) completely describes lifts of discrete \(\Z\)-modules to \(\S\). Apart from some functoriality subtleties, the picture is clear. There is no obstruction to existence of lifts and every abelian group lifts in an essentially unique way.

In Section~\ref{sec_moorescheme}, we define Moore schemes, an analog to Moore spectra, but, unlike for abelian groups, not all schemes lift. We establish the obstruction in Section \ref{sec_obstruction}.

\begin{defi}
We say that a connective \(\E_\infty\)-ring is a \emph{Moore ring} if its underlying spectrum is a Moore spectrum.
\end{defi}

As with Moore spectra, we use the notation \(\S_R\) to denote a Moore ring for a discrete ring \(R\). By definition, \(\S_R\) is then a lift of \(R\) to \(\S\). Since Moore spectra are unique, the question of whether \(R\) admits a lift becomes equivalent to existence of \(\E_\infty\)-ring structures on the Moore spectrum \(\S_R\) extending the ring structure on \(R=\pi_0\S_R\). A priori, there can exist multiple \(\E_\infty\)-structures on \(\S_R\) extending the same ring structure on \(R\).

\begin{rem}
Let \(\S_A\) and \(\S_B\) be Moore rings. Then base change of \(\S_A\otimes \S_B\) to \(\Z\) is isomorphic to \(A\otimes_\Z B\). Indeed, we can compute \((\S_A\otimes \S_B)\otimes\Z\simeq (\S_A\otimes \Z)\otimes_\Z(\S_B\otimes\Z)\simeq A\otimes_\Z B\). Hence, one could say that \(\S_A\otimes \S_B\) is a lift of \(A\otimes_\Z B\), even though the latter might not be discrete.
\end{rem}

\begin{lem} \label{lem_ratmoore}
Let \(R\) be a discrete ring. Suppose we are given lifts \(\S_{\comp R}\) of \(\comp R\) for each \(p\). Then the canonical map
\[\begin{tikzcd}
\big(\prod\S_{\comp R}\big)_\Q\ar[r] & \big(\prod\comp R\big)_\Q
\end{tikzcd}\]
is an equivalence.
\end{lem}

\begin{proof}
We have a fiber sequence
\[\begin{tikzcd}[sep=15]
\big(\prod(\S_{\comp R}\otimes \tau_{\geq 1}\S)\big)_\Q\ar[r] & \big(\prod\S_{\comp R}\big)_\Q\ar[r] & \big(\prod\comp R\big)_\Q.
\end{tikzcd}\]
It is therefore enough to show that the fiber is zero. Note that, for each \(N\geq 0\), we have \(\big(\prod_{p<N}(\S_{\comp R}\otimes \tau_{\geq 1}\S)\big)_\Q\simeq0\) as rationalization commutes with finite products and \(\tau_{\geq 1}\S\) is torsion by Serre finiteness theorem. Therefore, since \(\pi_n\big(\prod(\S_{\comp R}\otimes \tau_{\geq 1}\S)\big)_\Q\simeq\big(\prod\pi_n(\S_{\comp R}\otimes \tau_{\geq 1}\S)\big)_\Q\), the claim will follow if we show \(\pi_n(\S_{\comp R}\otimes \tau_{\geq 1}\S)\simeq 0\) for \(p >\hspace{-1mm}>0\). For this we can replace \(\tau_{\geq 1}\S\) by \(\tau_{[1,n]}\S\). As \(\S_{\comp R}\) is \(p\)-complete by Proposition \ref{prop_moorepcomp} and \(\tau_{[1,n]}\S\) vanishes after inverting finitely many primes, we get that \(\S_{\comp R}\otimes \tau_{[1,n]}\S\simeq 0\) for \(p>\hspace{-1mm}>0\), as required.
\end{proof}

\begin{prop}
Let \(R\) be a discrete ring. Suppose we are given lifts \(\S_{\comp R}\) of \(\comp R\) for each \(p\). Let \(\S_R\) denote the pullback
\[\begin{tikzcd}
\S_R \ar[r] \ar[d] & \prod \S_{\comp R}\ar[d] \\ R_\Q \ar[r] & \big(\prod \comp R\big)_\Q
\end{tikzcd}\]
of \(\E_{\infty}\)-rings. Then \(\S_R\) defines a lift of \(R\).
\end{prop}

\begin{proof}
By Lemma \ref{lem_ratmoore}, we have \((\S_R)_\Q\simeq R_\Q\), and, by construction, we have \(\comp{(\S_R)}\simeq \S_{\comp R}\). It then follows that \(\S_R\otimes \Z\simeq R\) as \(\E_\infty\)-rings, because their fracture squares are the same.
\end{proof}

Our next goal is to define Moore schemes.

\begin{prop}
Let \(f:A\to B\) be a morphism of connective \(\E_\infty\)-rings. If \(A\) is a Moore ring and \(f\) is flat, then \(B\) is a Moore ring.
\end{prop}

\begin{proof}
Since \(B\) is flat over \(A\), the base change \(B\otimes \Z\) is flat over \(A\otimes \Z\). As \(A\otimes \Z\) is discrete, \(B\otimes\Z\) is discrete as well.
\end{proof}

\begin{cor} \label{cor_Mloc}
Let \(\S_R\) be a Moore ring for \(R\) and let \(f\in R=\pi_0\S_R\). Then \(\S_R[f^{-1}]\) is a Moore ring for \(R[f^{-1}]\).
\end{cor}

\begin{cor} \label{cor_stalk}
Let \(\p\) be a point in \(\Spec \S_R\), that is a prime ideal of \(R\). Then the stalk \((\S_R)_\p\) at \(\p\) is a Moore ring for \(R_\p\).
\end{cor}

\begin{lem} \label{lem_Mglue}
Let \(A\) be a connective \(\E_{\infty}\)-ring and let \(f_i\in\pi_0A=:R\) be a collection of elements that generate the unit ideal. Assume that each \(A[f_i^{-1}]\) is a Moore ring for \(R[f_i^{-1}]\). Then \(A\) is a Moore ring for \(R\).
\end{lem}

\begin{proof}
By quasi-compactness, we can assume there are only finitely many \(f_i\)'s. Then \(R\), respectively \(A\), can be written as a finite limit of rings of the form \(R[f_{i_1}^{-1},...,f_{i_k}^{-1}]\), respectively \(A[f_{i_1}^{-1},...,f_{i_k}^{-1}]\). Since, by Corollary \ref{cor_Mloc}, \(A[f_{i_1}^{-1},...,f_{i_k}^{-1}]\) is a Moore ring for \(R[f_{i_1}^{-1},...,f_{i_k}^{-1}]\) and base change commutes with finite limits, we deduce that \(A\) is a Moore ring for \(R\).
\end{proof}

\begin{prop} \label{prop_Maff}
Let \(X\) be a connective spectral scheme and let \(\{\Spec \S_{R_i}\}\) be an affine open cover of \(X\), where each \(\S_{R_i}\) is a Moore ring. Then the global sections on every affine open of \(X\) form a Moore ring.
\end{prop}

\begin{proof}
Follows from Corollary \ref{cor_Mloc}, Lemma \ref{lem_Mglue}, and the Affine Communication Lemma.
\end{proof}

This leads to the following definition.

\begin{defi}
We define a \emph{Moore scheme} to be a connective spectral scheme for which global sections on each affine open are a Moore ring.
\end{defi}

We write \(X_\S\) to denote a Moore scheme for a discrete scheme \(X\). Then, by definition, \(X_\S\) defines a lift of \(X\) to \(\Spec \S\). Moreover, all lifts are of this form. Note that being a Moore scheme is a property, so Moore schemes form a full subcategory of connective spectral schemes.

\begin{rem}
By Lemma \ref{prop_Maff}, the property of being a Moore scheme is affine local.
\end{rem}

\begin{rem}[Stalk locality]
A derived scheme whose all stalks are discrete is itself discrete. Indeed, for each affine open \(\Spec R\), the sheaf \(\pi_nR\) would be nowhere supported when \(n>0\), hence would be zero. This implies, together with Corollary \ref{cor_stalk} and the fact that stalks commute with base change along \(\Spec \Z\to \Spec\S\), that a spectral scheme is a Moore scheme if and only if all its stalks are Moore rings.
\end{rem}

\begin{rem} \label{rem_achim}
In the introduction we mentioned a theorem of Achim Krause that \(\hat\delta\)-structures on a flat \(R\) are in one-to-one correspondence with \(H_{\infty}\)-structures on \(\S_R\). Together with the results of this section they imply that \(H_{\infty}\)-structures on flat Moore spectra satisfy Zariski descent. Étale descent should also be true.
\end{rem}

\begin{quest}
Do \(H_\infty\)-structures satisfy Zariski/étale descent (on \(\pi_0\)) in general?
\end{quest}

\subsection{Obstruction} \label{sec_obstruction}

In this section, we finally set up the obstruction for existence of lifts. We say obstruction in the sense of necessary condition.

\begin{thm}[{\cite[Proposition 2.41]{group_rings}}] \label{thm_obstruction}
Let \(\S_R\) be a Moore ring for a discrete ring \(R\). Then \(R\) admits a \(\hat\delta\)-structure functorial in \(\S_R\). That is, the base change functor \(-\otimes \Z:\CAlg\to\CAlg_\Z\) restricted to the full subcategory of Moore rings refines to a functor taking values in \(\Ring_{\hat\delta}\).
\end{thm}

\begin{con} \label{con_lift}
Let \(\S_R\) be a Moore ring for a discrete ring \(R\). Recall that \(\tau_{[0,1]}R^{\tCp}\simeq R/p\) and that \(\tau_{[0,1]}\) applied to the Tate valued Frobenius \(R\to R^{\tCp}\) agrees with the Frobenius on \(R\). We then get a commutative diagram
\[\begin{tikzcd}
\S_R \ar[r,"\varphi"] \ar[d] & \S_R^{\tCp} \ar[d] \\ R \ar[r,"\Frob"] & R/p
\end{tikzcd}\]
which, as the Frobenius is \(\Z\)-linear, factors through
\[\begin{tikzcd}
\S_R\otimes\Z \ar[r,"\varphi_\Z"] \ar[d,"\sim"] & \S_R^{\tCp}\otimes \Z \ar[d,"\can"] \\ R \ar[r,"\Frob"] & R/p.
\end{tikzcd}\]
Combining Proposition \ref{prop_torTate} with Corollary \ref{cor_comp} yields \(\S_R^{\tCp}\otimes \Z\simeq \comp R\), and so \(\phi_\Z\) defines a \(\Z\)-linear lift
\[\begin{tikzcd}
& \comp R \ar[d,"\can"] \\ R \ar[r,"\Frob"] \ar[ur,"\phi_\Z"] & R/p
\end{tikzcd}\]
of the Frobenius on \(\comp R\). Moreover, this construction is functorial in \(\S_R\).
\end{con}

\begin{proof}[Proof of Theorem \ref{thm_obstruction}]
In view of Construction \ref{con_lift}, it suffices to show that the obtained Forbenius lift refines to a lift in \(\Ring^\an\). Combining Corollary \ref{cor_2trun} (for \(S=\comp R\)) with Corollary \ref{cor_pi1tor} we get that \(\phi_\Z\) admits a unique refinement to an animated ring map. Then, applying Corollary \ref{cor_2trun} again (for \(S=R/p\)), we see that the homotopy exhibiting \(\phi_\Z\) as a Frobenius lift admits a unique refinement as well.
\end{proof}

We have set up \(\hat\delta\)-schemes in such a way that Theorem \ref{thm_obstruction} glues.

\begin{thm} \label{thm_schemes}
Let \(X_\S\) be a Moore scheme for a discrete scheme \(X\). Then \(X\) admits a \(\hat\delta\)-scheme structure functorial in \(X_\S\).
\end{thm}

\begin{proof}
Base change along \(\Spec \Z\to\Spec\S\) does not modify the underlying topological space of \(X_\S\). Therefore, it is computed on affine opens by applying \(-\otimes\Z\) to global sections, which by assumption form a Moore ring. Hence, by Theorem \ref{thm_obstruction} and Example \ref{ex_loc}, we get a \(\hat\delta\)-scheme structure on \(X\). Theorem \ref{thm_obstruction} also yields functoriality in maps between affine Moore schemes. This is enough, because any map of schemes is glued from maps on affine opens.
\end{proof}

Thus, we obtain a necessary condition for a scheme \(X\) to admit a lift to \(\Spec\S\). Namely, there must exist a \(\hat\delta\)-structure on \(X\). Moreover, given two lifts \(X_\S\) and \(Y_\S\), a necessary condition for a map \(X\to Y\) to lift to \(X_\S\to Y_\S\) is that it can be refined to a \(\hat\delta\)-scheme morphism.

\begin{war}
The proof is very specific to \(\S\) and \(\Z\), and to the fact that \(X\) is discrete.
\end{war}

\subsection{Some \(\hat\delta\)-structure computations} \label{sec_computations}

We will now investigate how certain common algebraic constructions interact with the Tate valued Frobenius. We will also identify the resulting \(\hat\delta\)-structures. We start with the most important computation.

\begin{ex}[Monoid rings] \label{ex_frobmon}
Let \(M\) be a monoid. We can form the monoid ring \(\S[M]=\bigoplus_M \S\), which is a Moore ring for \(\Z[M]\). We want to describe the Tate Frobenius of \(\S[M]\).

By the fact that \((-)^{tC_p}\) is additive and vanishes on induced representations, we get \(T_p(\S[M])=(\bigoplus_{M^{\times p}}\S)^{tC_p}\simeq (\bigoplus_{M}\S)^{tC_p} = \S[M]^{tC_p}\), with the isomorphism given by the diagonal embedding \(M\subset M^{\times p}\). Under this identification, the Tate diagonal \(\S[M]\to \S[M]^{tC_p}\) agrees with the canonical map. Therefore, we learn that the Tate valued Frobenius \(\S[M]\to \S[M]^{tC_p}\) factors as \(\S[M] \to \S[M]^{tC_p} \to \S[M]^{tC_p}\), where the first map is the canonical one and the second map is given by applying \((-)^{tC_p}\) to \(\alpha:\S[M]\to \S[M]\) induced by \([m]\to [pm]\).

Moreover, by Proposition \ref{prop_torTate}, the map \(\alpha^{tC_p}\) can be identified with the map \(\comp\alpha:\comp{\S[M]}\to \comp{\S[M]}\). Therefore, the \(\hat\delta\)-structure on \(\Z[M]\) is given, on each \(p\)-completion, by \(\Z[M]\to \comp{\Z[M]}\) induced by \([m]\to [pm]=[m]^p\). That is \(\delta([m])=0\).
\end{ex}

\begin{ex}[Colimits]
Let \(\{\S_{A_i}\}\) be a diagram of Moore rings. Then the colimit \(\S_A=\colim\, \S_{A_i}\) is a lift of \(A=\colim\, A_i\) (even if \(A\) is not discrete). The Tate valued Frobenius for \(\S_A\) factors as
\[\begin{tikzcd}[sep=15] \S_A\ar[r] & \colim\, \S_{A_i}^{tC_p} \ar[r] & \S_A^{tC_p}, \end{tikzcd}\]
where the first map is the colimit of Tate valued Frobenii for \(\S_{A_i}\) and the second map is just the colimit assembly map. Therefore, after base change, we recover the natural \(\hat\delta\)-structure on a colimit of animated \(\hat\delta\)-rings, in particular tensor products.
\end{ex}

\begin{ex}[Localizations] \label{ex_loc}
Let \(\S_A\) be a Moore ring for a discrete ring \(A\) and let \(f\in A\). Then \(\S_{A}[f^{-1}]=:\S_{A[f^{-1}]}\) is a Moore spectrum for \(A[f^{-1}]\). By naturality of the Tate valued Frobenius we get a commutative diagram
\[\begin{tikzcd}
\S_A \ar[r] \ar[d] & \S_A^{tC_p} \ar[d] \\
\S_{A[f^{-1}]}\ar[r] & \S_{A[f^{-1}]}^{tC_p}.
\end{tikzcd}\]
Hence, after base change to \(\Z\), we recover the \(\hat\delta\)-structure on \(A[f^{-1}]\) described in Construction \ref{con_loc}.
\end{ex}

\begin{ex}[Completions]
Let \(\S_A\) be a Moore ring for \(A\). Then \(\comp{(\S_A)}=:\S_{\comp A}\) is a Moore ring for \(\comp A\) (c.f.\@ Corollary \ref{cor_comp}). By \cite[Lemma I.2.9]{TC} and naturality, the Tate valued Frobenius for \(\S_A\) factors as
\[\begin{tikzcd}[sep=15] \S_A\ar[r] & \S_{\comp A} \ar[r] & \S_A^{tC_p}\simeq \S_{\comp A}^{tC_p}, \end{tikzcd}\]
where the first map is \(p\)-completion and the second is the Tate valued Frobenius for \(\S_{\comp A}\). Hence, after base change to \(\Z\), we see that the induced \(\hat\delta\)-structure on \(\comp A\) is just the \(\delta\)-structure on \(\comp A\) coming from the \(\hat\delta\)-structure on \(A\).
\end{ex}

For definitions and basic properties of square-zero extensions we refer to \cite[Section 7.4]{HA}.

\begin{ex}[Split square-zero extensions] \label{ex_sqzero}
Let \(\S_A\) be a Moore ring for \(A\) and let \(\S_M\) be a lift of an \(A\)-module \(M\) to an \(\S_A\)-module. Consider the split square-zero extension \(\S_{A}\oplus \S_M\). Base changing to \(\Z\) yields the split square-zero extension \(A\oplus M\) of \(A\) by \(M\). Since \(T_p\) is an exact functor, we get that the map
\[\begin{tikzcd}[sep=15] (\S_A^{\otimes p})^{tC_p}\oplus(\S_M^{\otimes p})^{tC_p}\ar[r] & T_p(\S_A\oplus \S_M) \end{tikzcd}\]
is an equivalence. Since the extension is square-zero, the multiplication map \(\S_M^{\otimes p}\to \S_M\) is nullhomotopic, and so the Tate valued Frobenius for \(\S_A\oplus \S_M\) factors as
\[\begin{tikzcd}[sep=15] \S_A\oplus \S_M \ar[r] & \S_A \ar[r,"\varphi"] & \S_A^{tC_p} \ar[r,"{(1,0)}"] &[3mm] (\S_A\oplus \S_M)^{tC_p}. \end{tikzcd}\]

From this we deduce that the induced \(\hat\delta\)-structure on \(A\oplus M\) satisfies \(\phi(a,m)=(\phi(a),0)\). This yields, at least if \(A\) and \(M\) are \(p\)-torsion free, the formula \(\delta(a,m)=(\delta(a),-a^{p-1}m)\).
\end{ex}

\section{Examples} \label{chap_ex}

We will now demonstrate some examples of rings and schemes whose lifting problems can be successfully studied with our obstruction.

\subsection{Étale algebras}

We begin by discussing étale algebras. Let us recall the definition.

\begin{defi}[{\cite[Definition 7.5.0.4]{HA}}]
An \(\E_{\infty}\)-algebra \(A\) is \emph{étale} if it is flat over \(\S\) and the ring \(\pi_0A\) is étale over \(\Z\).
\end{defi}

We have the following important theorem.

\begin{thm}[{\cite[Theorem 7.5.0.6]{HA} \label{thm_etale}}\hspace{-1mm}]
The assignment \(A\mapsto \pi_0A\) induces an equivalence between étale algebras over \(\S\) and étale algebras over \(\Z\).
\end{thm}

Let \(A\) be étale over \(\S\) and let \(R=\pi_0A\). Since \(A\) is flat, we have \(A\otimes\Z\simeq R\), so \(A\) defines a lift of \(R\). By Proposition \ref{prop_flat}, any Moore ring for \(R\) is étale over \(\S\). Theorem \ref{thm_etale} then implies that \(R\) has a unique lift.

\begin{prop} \label{prop_undelta}
Let \(R\) be an étale algebra over \(\Z\). Then \(R\) admits a unique \(\hat\delta\)-structure.
\end{prop}

\begin{proof}
Follows by applying, for each prime \(p\), Lemma 2.18 of \cite{prismatic} in the case \(A=\Z\), \(I=(p)\), \(B=\comp R\).
\end{proof}

We conclude that for an étale algebra over \(\Z\) there exist both a unique lift to \(\S\) and a unique \(\hat\delta\)-structure, necessarily induced by that lift.

\subsection{Rings of integers}

Let \(K\) be an algebraic number field, that is a finite extension of \(\Q\), and consider its associated ring of integers \(\O\). In this section, we will completely determine the obstruction for \(\O\) to admit a lift to the sphere spectrum. We will implicitly use some basic results from local class field theory. They can all be found, for example, in \cite{local}. The special case of adjoining roots of unity has already been considered in \cite{roots}.

Let \(\mathcal P\) denote the set of all ramified primes in \(\O\). For any prime \(p\) we have that \(\Z_p\subset \comp{(\O)}\) factors as \(\Z_p\subset \O'\subset \comp{\O}\), where the extension \(\Z_p\subset \O'\) is unramified and \(\O'\subset \comp{\O}\) totally ramified. That is, if \(p\notin\mathcal P\), we have \(\O'=\comp{\O}\), and if \(p\in\mathcal P\), then we have an isomorphism \(\comp{\O}\simeq \O'[x]/E(x)\) with \(E\in\O'[x]\) an Eisenstein polynomial for the maximal ideal \(\mathfrak m\) of \(\O'\). Since \(\Z_p\subset \O'\) is unramified we just have \(\mathfrak m = p\O'\).

Now, suppose \(\O\) admits a lift to the sphere. Then, by Theorem \ref{thm_obstruction}, for \(p\in\mathcal P\), there exists a \(\delta_p\)-structure on \(\O'[x]/E(x)\). Let \(\phi\) be the associated Frobenius lift. Since \(\phi(x)=x^p+p\delta(x)\) and \(\phi(p)=p\), we get that both \(\phi\big(E(x)\big)-\phi\big(E(0)\big)-x^{p^2}\) and \(E(x)^p-x^{p^2}\) are contained in the ideal \((p^2,px)\). Hence, we deduce that modulo \((p,x)\) we have
\[0=\delta\big(E(x)\big)=\frac{\phi\big(E(x)\big)-E(x)^p}{p}=\frac{\phi\big(E(0)\big)}{p}.\]
But \(E\) is Eisenstein, and so \(E(0)=pu\), with \(u\) a unit. Thus, \(\delta\big(E(x)\big)=\phi(u)\) is a unit as well, which contradicts our assumption.

\begin{cor}
No ring of integers of an algebraic number field lifts to \(\S\).
\end{cor}

\begin{proof}
By the discussion above, if \(K\) is ramified, then \(\O\) cannot admit a lift. But, due to Minkowski's bound, \(\Q\) admits no unramified extensions.
\end{proof}

On the other hand, after inverting the ramified primes, the ring \(\O[\mathcal P^{-1}]\) becomes unramified, hence étale. So \(\O[\mathcal P^{-1}]\) admits a lift by Theorem \ref{thm_etale}.

Uniqueness of the \(\hat\delta\)-structure follows once again by Proposition \ref{prop_undelta}, but in this case we can compute it more precisely.

\begin{prop}
The ring \(\O[\mathcal P^{-1}]\) admits a unique \(\hat\delta\)-structure.
\end{prop}

\begin{proof}
If \(p\in \mathcal P\), then \(\comp{\O[\mathcal P^{-1}]}=0\), so there the \(\delta_p\)-structure is unique. If \(p\notin\mathcal P\), then \(\Z_p\subset \comp{\O[\mathcal P^{-1}]}=\comp\O\) is unramified. For each \(n\) there is a unique unramified extension of \(\Q_p\) of degree \(n\). Moreover, it is given by adjoining the \((p^n-1)\)-th root of unity \(\zeta\). Thus, we get that \(\comp\O\simeq\Z_p[\zeta]\) for some \(n\). We will show that \(\delta(\zeta)=0\), from which uniqueness will follow.

Note that \(\zeta^m=1\) for some \(m\) coprime to \(p\). Therefore, we get
\[0=\delta(\zeta^m)=\frac{\phi(\zeta)^m-\zeta^{mp}}{p}=\frac{(\zeta^p+p\delta(\zeta))^m-\zeta^{mp}}p.\]
Unravelling this expression we see that \(\delta(\zeta)\) must be infinitely \(p\)-divisible, because \(m\) and \(\zeta\) are invertible in \(\Z_p[\zeta]\). This means \(\delta(\zeta)=0\), as wanted.
\end{proof}

\begin{ex}[{see also \cite[Example I.3.49]{Schwede}}]
Let \(p\) be a prime number and \(q=p^n\) such that \(q>2\). Consider \(\Z[\zeta_q]:=\Z[x]/f\) with \(f(x)=1+x^{p^{n-1}}+...+x^{(p-1)p^{n-1}}\). Note that \(f(x+1)\) is an Eisenstein polynomial, hence by the previous discussion, we get that \(\Z[\zeta_q]\) does not lift to the sphere. On the other hand, after inverting \(p\), one can check that
\[\Z[\zeta_q,\frac1p]\simeq \Z[C_q,\frac1{p-f(x)}],\]
where \(C_q\) denotes the cyclic group of order \(q\) and we identify \(x\) with its generator. Thus, \(\Z[\zeta_q,\frac1p]\) lifts as \(\S[C_q,\frac1{p-f(x)}]\). We also see that the corresponding \(\hat\delta\)-structure is given by \(\delta(x)=0\) as expected (c.f.\@ Example \ref{ex_frobmon}).
\end{ex}

Apart from just rings of integers one could consider general orders (i.e.\@ finite reduced \(\Z\)-algebras). We compute the following.

\begin{ex}
We will show that \((\Z[ni])_2^\wedge\) does not lift to \(\S\) for any integer \(n\ne 0\). Let us denote \(f=\delta(ni)\). From the formula \(\delta\left((ni)^2\right)=\delta(-n^2)=\frac{-n^2-n^4}{2}\) and the multiplicative relation for \(\delta\) we get
\[2f^2-2n^2f+\frac{n^4+n^2}2=0,\]
which has the solution \(f=\frac{2n^2\pm 2ni}4\). As \(\frac12ni\) does not belong to \((\Z[ni])_2^\wedge\), there can be no such delta-structure.
\end{ex}

\subsection{Torsion rings}

Recall the following lemma.

\begin{lem} \label{lem_torvan}
Let \(A\) be a \(\delta_p\)-ring such that \(p^n=0\) for some \(n\). Then \(A=0\).
\end{lem}

\begin{proof}
We have
\[0=\delta(p^n)=\frac{p^n-p^{pn}}{p}=p^{n-1}.\]
Hence, by induction, we get \(1=0\), as desired.
\end{proof}

From this we can deduce a \(\hat\delta\)-ring analog.

\begin{lem}
Let \(A\) be a torsion \(\hat\delta\)-ring. Then \(A=0\).
\end{lem}

\begin{proof}
Suppose \(A\) is \(n\)-torsion for some \(n\in\Z\backslash\{0\}\). If \(n=1\) the result is obvious. If not, then let \(p\) be a prime divisor of \(n\) and \(p^m\) the maximal \(p\)-power that divides \(n\). By Proposition \ref{prop_tauhat}, \(\pi_0\comp A\) is a \(\delta_p\)-ring. Moreover, since \(p\)-completion inverts other primes, we have \(p^m=0\) in \(\pi_0\comp A\). Thus, \(\pi_0\comp A=0\) by Lemma \ref{lem_torvan}. This implies \(\comp A\simeq 0\), and so \(A/p\simeq \comp A/p\simeq 0\), from which we deduce that \(p\) acts invertibly on \(A\). Arguing in the same fashion for each prime divisor of \(n\), we conclude that \(n\) is invertible, so \(A=0\).
\end{proof}

\begin{cor}
Let \(A\ne 0\) be a torsion ring. Then \(A\) admits no lift to the sphere.
\end{cor}

In particular, we recover the following result, also implied by \cite[Remark 4.3]{moore}.

\begin{cor} \label{cor_smodn}
The Moore spectrum \(\S/n\) admits no \(\E_\infty\)-ring structure.
\end{cor}

\subsection{Group schemes} \label{sec_group}

For group schemes we are not merely interested in the existence of a schematic lift, but rather of a lift as a spectral group scheme. By naturality of our obstruction, if such a lift exists, then the group scheme in question admits a \(\hat\delta\)-group scheme structure, that is a \(\hat\delta\)-scheme structure such that the multiplication, the inverse, and the unit morphisms are \(\hat\delta\)-scheme maps.

\begin{ex}
As a first application, we show that \(\G_a\) admits no lift to \(\Spec\S\) as a group scheme.

Suppose to the contrary that it does. Then the comultiplication \(\Z[a]\to\Z[x,y]\) sending \(a\) to \(x+y\) would be a \(\hat\delta\)-ring map for some \(\hat\delta\)-ring structure on \(\Z[a]\), which is given, for each \(p\), by some assignment \(\delta(a)=f(a)\in\comp{\Z[a]}\). We get the relation
\[f(x+y)=\delta(x+y)=f(x)+f(y)+\frac{x^p+y^p-(x+y)^p}{p}.\]
Modulo \(p\) those are polynomials and the left side cannot have a non-zero \(x^{p-1}y\)-term, which the right side does. Hence, no \(f\) satisfies this equation.
\end{ex}

\begin{ex}
On the other hand, we can show that there exists a lift and a unique \(\hat\delta\)-group scheme structure on \(\G_m\). Indeed, arguing as before, we get the equation
\[f(xy)=\delta(xy)=x^pf(y)+f(x)y^p+pf(x)f(y).\]
Suppose \(f\ne 0\). Then there exists \(N\geq1\) such that \(f\) is divisible by \(p^{N-1}\) but not \(p^{N}\). Modulo \(p^{N}\), we then get \(f(xy)=x^pf(y)+f(x)y^p\), which, for degree reasons, can only be satisfied if \(f(a)=\lambda a^p\) modulo \(p^N\), which actually implies that \(\lambda=0\) modulo \(p^N\), a contradiction.

It is clear that \(\phi(a)=a^p\) respects the comultiplication, the antipode, and the augmentation, so \(\G_m\) is a \(\hat\delta\)-group scheme. In fact, it is also clear that the spectral group scheme \(\Spec \S[\Z]\) lifts \(\G_m\), because the group structure is induced by monoid maps.
\end{ex}

\begin{ex} \label{ex_matr}
We now show that, for \(n>1\), the group scheme \(\text{GL}_n\) has no \(\hat\delta\)-monoid scheme structure, and so does not lift to the sphere. Suppose to the contrary that there is one. Then \(\Delta(a_{ij})=\sum_k x_{ik}y_{kj}\) must be a \(\hat\delta\)-map, and so we must have
\[\Delta\big(\delta(a_{ij})\big)=\delta\big({\textstyle \sum_k x_{ik}y_{kj}}\big)=\frac{\sum_k \phi(x_{ik})\phi(y_{kj})-(\sum_k x_{ik}y_{kj})^p}{p}.\]
Reducing modulo \(p\) we get
\[\sum\nolimits_k \Big(x_{ik}^p\delta(y_{kj})+\delta(x_{ik})y_{kj}^p\Big)-\Delta\big(\delta(a_{ij})\big)=\frac{\sum_k x_{ik}^py_{kj}^p-(\sum_k x_{ik}y_{kj})^p}{p}.\]
But the right side has a non-zero \(x_{ik}y_{kj}(x_{it}y_{tj})^{p-1}\)-term for some \(k\ne t\), which is impossible on the left. Hence, there is no such \(\delta\).
\end{ex}

\begin{rem}
The fact that \(\G_a\) does not admit a lift implies that ,\hspace{-0.33mm},very commutative" elements of an \(\E_\infty\)-ring are not closed under addition. On the other hand, since \(\G_m\) does, they are closed under multiplication.
\end{rem}

\subsection{Projective schemes} \label{sec_proj}

In this section, we are interested in liftability of a closed subscheme \(Z\) of the projective space \(\P^n\) together with its embedding \(Z\subset \P^n\). That is, we are looking for \(\hat\delta\)-structures on \(Z\), but restrict ourselves only to those which come from a \(\hat\delta\)-structure on the whole \(\P^n\). Hence, the first step is to identify possible \(\hat\delta\)-structures on \(\P^n\).

Suppose we are given a \(\hat\delta\)-structure on \(\A^1\). By taking fiber products we also get one on each \(\A^n\). We will say that the \(\hat\delta\)-structure on \(\A^1\) is \emph{coherent} if the following conditions hold:
\begin{enumerate}
\item The \(\hat\delta\)-structures on \(\Spec\Z[\frac{x_0}{x_i},...,\frac{x_n}{x_i}]=\A^n\) glue to a \(\hat\delta\)-structure on \(\P^n\).
\item The map \(\A^{n+1}\backslash\{0\}\to \P^n\) is a morphism of \(\hat\delta\)-schemes.
\end{enumerate}

Recall that there is a canonical \(\hat\delta\)-structure on \(\A^1\) induced by the lift \(\Spec \S[x]\), which we denote by \(\A^1_{\flat\S}\) (to distinguish from the non-flat \(\Spec\S\{x\}\)). It is given, on each \(p\)-completion, by \(\delta(x)=0\).

\begin{prop}
The canonical \(\hat\delta\)-structure on \(\A^1\) is the only one which is coherent.
\end{prop}

\begin{proof}
Fix a coherent \(\hat\delta\)-structure on \(\A^1\). For a given prime \(p\) it is determined by an assignment \(\delta(x)=f(x)\in\comp{\Z[x]}\). Conditions (1) and (2) then imply the formula:
\[f\Big(\frac{x_i}{x_j}\Big)=\delta\Big(\frac{x_i}{x_j}\Big)=\frac{\frac{\phi(x_i)}{\phi(x_j)}-\big(\frac{x_i}{x_j}\big)^p}{p}=\frac{f(x_i)}{\phi(x_j)}-\Big(\frac{x_i}{x_j}\Big)^p\frac{f(x_j)}{\phi(x_j)},\]
which can be rewritten as
\[x_j^{p}\phi(x_j)f\Big(\frac{x_i}{x_j}\Big)=x_j^pf(x_i)-x_i^pf(x_j).\]
If \(f=0\), then this is satisfied. If not, then there exists \(N\geq 1\) such that \(f\) is divisible by \(p^{N-1}\), but not by \(p^{N}\). Since \(\phi(x_j)=x_j^p+pf(x_j)\), we get, modulo \(p^N\), the equality
\[x_j^{2p}f\Big(\frac{x_i}{x_j}\Big)=x_j^pf(x_i)-x_i^pf(x_j).\]
Both sides are polynomials: left side of degree \(2p\), right side of degree \(2p-1\) when \(\deg f= p\) and \(\deg f+p\) when \(\deg f\ne p\). This is a contradiction.

Coherence of the canonical \(\hat\delta\)-structure follows from Construction \ref{con_projlift}.
\end{proof}

We will now construct a lift for \(\P^n\). Our treatment is somewhat informal, for a more detailed account see \cite[Construction 5.4.1.3]{SAG}.

\begin{con} \label{con_projlift}
Recall one of the standard constructions of \(\P^n\): we consider monoid rings \(\Gamma(U_i)=\Z[x_{0/i},...,\hat x_{i/i},...,x_{n/i}]\) (with \(x_{i/i}\) omitted) and want to glue the corresponding affine schemes along distinguished opens \(U_{ij}=D(x_{j/i})\subset U_i\) and isomorphisms \(U_{ij}\to U_{ji}\) given by \(x_{t/j}\mapsto x_{t/i}x_{j/i}^{-1}\). The cocycle condition then takes form of the equality \(x_{t/i}x_{k/i}^{-1}=(x_{t/i}x_{j/i}^{-1})(x_{k/i}x_{j/i}^{-1})^{-1}\). Note that both the localizations \(U_i\to U_{ij}\) and isomorphisms \(U_{ij}\to U_{ji}\) are induced by monoid maps. Moreover, the cocycle condition holds on the level of those maps. Taking those monoid maps and applying \(\Spec\S[-]\) instead of \(\Spec\Z[-]\) yields a gluing datum of spectral schemes with all higher cocycle coherences automatically satisfied. Hence, we get a spectral scheme \(\P^n_{\flat\S}\), which by construction lifts \(\P^n\).

Recall also that the map \(\A^{n+1}\backslash\{0\}\to \P^n\) is glued from maps on affine pieces \(D(x_i)\to U_i\) induced by inclusions \(\Z[x_{0/i},...,\hat x_{i/i},...,x_{n/i}]\to \Z[x_{0},..., x_{i}^{\pm},...,x_{n}]\) sending \(x_{t/i}\) to \(x_tx_i^{-1}\). Those again are monoid maps, and so glue on the spectral level to \(\A^{n+1}_{\flat\S}\backslash\{0\}\to\P^n_{\flat\S}\).
\end{con}

\begin{rem}
Similarly as in Construction \ref{con_projlift}, one can make flat lifts of toric varieties, because their gluing datum is governed by combinatorics of fans. For a precise statement we refer to \cite[Construction 3.1.2]{Bai}.
\end{rem}

Our question now becomes: when does a closed subscheme \(Z\subset \P^n\) lift to a closed subscheme of \(\P^n_{\flat\S}\)? We describe some concrete examples.

\begin{ex}[Elliptic curves]
Consider \(Z\subset \P^2\) cut out by \(y^2z=x^3+axz^2+bz^3\) with \((a,b)\ne(0,0)\). We will show that the canonical \(\hat\delta\)-structure on \(\P^n\) does not descend to \(Z\). Hence, that the embedding does not lift.

On the affine patch \(D_+(z)\) the equation becomes \(f(x,y)=y^2-x^3-ax-b=0\), and so, on the \(2\)-completion, we would have that
\[\delta\big(f(x,y)\big)=\frac{f(x^2,y^2)-f(x,y)^2}{2}\]
is divisible by \(f(x,y)\). But, expanding the right side and then substituting \(y^2=x^3+ax+b\) we get
\[ax^4+bx^3+\frac{a^2-a}2x^2+abx+\frac{b^2-b}2\]
which is zero only if \(a=b=0\).

On the other hand, cuspidal curves \(y^2z^{m-2}=x^m\) lift by Example \ref{ex_coeq}.
\end{ex}

\begin{ex} \label{ex_coeq}
Consider two monic monomials \(f,g\in\Z[x_0,...,x_n]\) such that \(\deg f = \deg g=:d\) or \(g=0\). We claim that the closed subscheme \(Z\subset \P^n\) cut out by the homogenous equation \(f-g\) lifts to \(\P^n_{\flat\S}\).

This follows from a similar argument as in Construction \ref{con_projlift}. The key is that, for each \(i\), both \(\frac{f}{x_i^d}\) and \(\frac{g}{x_i^d}\) are monomials, and so define ring maps \(f_i,g_i:\S[t]\to\S[\frac{x_0}{x_i},...,\frac{x_n}{x_i}]\). Then, on \(D_+(x_i)\), we can lift the ring \(A_i=\Z[\frac{x_0}{x_i},...,\frac{x_n}{x_i}]/(\frac{f}{x_i^d}-\frac{g}{x_i^d})\) via the pushout
\[\begin{tikzcd}
\S[t,\frac{x_0}{x_i},...,\frac{x_n}{x_i}] \ar[d,"{(g_i,1)}"] \ar[r,"{(f_i,1)}"] & \S[\frac{x_0}{x_i},...,\frac{x_n}{x_i}] \ar[d] \\
\S[\frac{x_0}{x_i},...,\frac{x_n}{x_i}] \ar[r] & \S_{A_i}.
\end{tikzcd}\]
Since, once more, everything is induced from maps of monoids, the corresponding affine spectral schemes glue to the desired lift.
\end{ex}

\appendix

\section{Moore spectra and bounded Tor-amplitude} \label{sec_Mtor}

In this section of the appendix, we recall the theory of Moore spectra and bounded Tor-amplitude. For proofs of basic properties of Moore spectra we refer to \cite[Section 6.3]{Schwede}.

\begin{adefi}
A Moore spectrum for an abelian group \(A\) is a connective spectrum \(\S_A\) such that \(\S_A\otimes\Z\simeq A\).
\end{adefi}

\begin{arem}
Moore spectra are unique up to equivalence, but do not define a functor \(\text{Ab}\to \text{Sp}\), so our notation is a little abusive.
\end{arem}

\begin{aprop} \label{prop_flat}
The Moore spectrum \(\S_A\) is flat over \(\S\) if and only if \(A\) is flat over \(\Z\).
\end{aprop}

\begin{aprop}[{\cite[Corollary 3.3]{barthel}}] \label{prop_moorepcomp}
Let \(X\) be a bounded below spectrum. Then \(X\) is \(p\)-complete if and only if \(X\otimes \Z\) is \(p\)-complete.
\end{aprop}

\begin{acor} 
Let \(\S_A\) be a Moore spectrum. Then \(\S_A\) is \(p\)-complete if and only if \(A\) is.
\end{acor}

\begin{acor} \label{cor_comp}
Let \(\S_A\) be a Moore spectrum. Then \(\comp{(\S_A)}\otimes \Z\to \comp A\) is an isomorphism.
\end{acor}

\begin{adefi}[{\cite[Definition 7.2.4.21]{HA}}]
We say that \(X\in \Sp\) has \emph{Tor-amplitude} in \([a,b]\) if \(X\) is \(a\)-connective and, for any \(A\in \Ab\), the tensor product \(X\otimes A\) is \(b\)-truncated. We say that \(X\) has \emph{bounded Tor-amplitude} if it has Tor-amplitude in \([a,b]\) for some integers \(a\) and \(b\).
\end{adefi}

\begin{arem} \label{rem_cell}
Equivalently, \(X\) has bounded Tor-amplitude if and only if it can be built as a cell complex using cells of bounded dimension, but not necessarily finitely many.
\end{arem}

\begin{aex}
A Moore spectrum \(\S_A\) has Tor-amplitude in \([0,1]\). Indeed, for any abelian group \(B\), we have \(\S_A\otimes B \simeq A\otimes_\Z B\), which is 1-truncated.
\end{aex}

The following proposition can be interpreted as a generalization of the Segal conjecture and can be found, for example, in \cite{yuan}. We give our own proof.

\begin{aprop} \label{prop_torTate}
Let \(X\) be of bounded Tor-amplitude. We equip \(X\) with the trivial \(C_p\)-action. Then the canonical map \(X\to X^{hC_p}\to X^{\tCp}\) exhibits \(X^{tC_p}\) with the \(p\)-completion of \(X\).
\end{aprop}

\begin{proof}
Since Tate construction is exact, using Remark \ref{rem_cell}, we can reduce to the case \(X=\bigoplus \S\). For \(X=\S\) this follows from \cite[Theorem III.1.7]{TC} and the equivalence \(\S^{tC_p}\simeq T_p(\S)\). The classifying space \(BC_p\) is \(\aleph_1\)-small, so both \((-)^{\tCp}\) and \(p\)-completion commute with \(\aleph_1\)-filtered colimits. Therefore, we can further reduce to the case \(X=\bigoplus_\N\S\).

Consider the set \(\N^{\times p}\) with the \(C_p\) action given by translation. The diagonal \(\N\to \N^{\times p}\) admits a \(C_p\)-equivariant retraction that sends \((n_1,...,n_p)\) to \(\lfloor\frac{\sum n_i}{p}\rfloor\). This defines a retraction of the map \(X^{\tCp}\to (X^{\otimes p})^{\tCp}\) from Example \ref{ex_frobmon}. By \cite[Theorem III.1.7]{TC}, the composite \(X\to X^{\tCp}\to (X^{\otimes p})^{\tCp}\) can be identified with the \(p\)-completion of \(X\). Since the last map splits, we get that \(X^{\tCp}\) is \(p\)-complete, and so, by the universal property, is in fact an equivalence.
\end{proof}

\section{\(\Z\)-linear vs.\@ animated maps} \label{sec_ani}

We establish a technical lemma about the discrepancy between animated and \(\Z\)-linear morphisms of rings. It is an important ingredient for the proof of Theorem \ref{thm_obstruction}, but does not give much conceptual insight. That is why we decided to put it in the appendix. The proof uses the cotangent complex formalism of \cite{HA}, \cite{SAG}, and can be safely skipped by the reader. The ideas in this section are due to Thomas Nikolaus.

Recall the following result.

\begin{aprop}[{\cite[Proposition 25.1.2.2]{SAG}}]
The forgetful functor induces an equivalence of categories
\[\Ring^\an_\Q\to \CAlg^\cn_{\Q}.\]
\end{aprop}

In contrast, the forgetful functor \(\Ring^\an\to\CAlg^\cn_{\Z}\) is neither full nor faithful. Our goal is to get a better handle on the difference between respective mapping spaces after inverting more and more primes.

We let \(p_n\) denote the \(n^{\text{th}}\) prime number. We also use the convention \(p_0=1\) and \(p_{\infty}=\infty\). We write \(\Z[p_{\leq n}^{-1}]\) for the ring of integers with all primes \(\leq p_n\) inverted. In particular, we have \(\Z[p_{\leq 0}^{-1}]=\Z\) and \(\Z[p_{\leq\infty}^{-1}]=\Q\).

To avoid confusion, we will write \(\Map_\an\) for the mapping space in \(\Ring^\an\) and \(\Map_\Z\) for the mapping space in \(\CAlg^\cn_{\Z}\). The forgetful functor induces a comparison map \(\Map_\an(-,-)\to\Map_\Z(-,-)\).

\begin{aprop} \label{prop_anvs}
Let \(R\) and \(S\) be animated rings and let \(0\leq n\leq \infty\). Assume that \(S\) is an \(m\)-truncated \(\Z[p_{\leq n}^{-1}]\)-algebra. If \(m\leq 2p_{n+1}-3\), then the comparison map
\[\Map_\an(R,S)\to\Map_\Z(R,S)\]
has contractible fibers and is an equivalence if \(\pi_{2p_{n+1}-3}S\) is \(p_{n+1}\)-torsion free.
\end{aprop}

The proof relies on the following lemma. Another approach would be to identify first \(p\)-torsion in the integral homology of symmetric groups and compare the free \(\Z\)-algebra \(\Z\{x\}=\bigoplus \Z_{h\Sigma_n}\) with the free animated ring \(\Z[x]\) (see \cite[Lemma 2.39]{group_rings}).

\begin{alem} \label{lem_cot}
Let \(R\) be an animated ring over \(\Z[p_{\leq n}^{-1}]\) and let \(0\leq n \leq \infty\). Then the fiber of the comparison map
\[\rho_R:L_{R/\Z}\to L^\alg_R\]
between the topological and the algebraic cotangent complex of \(R\) is \((2p_{n+1}-3)\)-connective and \(\pi_{2p_{n+1}-3}\big(\fib(\rho_R)\big)\) is \(p_{n+1}\)-torsion.
\end{alem}

\begin{proof}
By \cite[Section 25.3.5]{SAG} there is an \(\E_1\)-ring \(R^+\) such that \(L^\alg_R\simeq R\otimes_{R^+}L_{R/\Z}\) and that \(\rho_R\) is given by \(\gamma\otimes_{R^+}L_{R/\Z}\), where \(\gamma\) is a map \(R^+\to R\). Thus, the fiber of \(\rho_R\) can be identified with \(\fib(\gamma)\otimes_{R^+}L_{R/\Z}\). By a result of Schwede (\cite[Proposition 25.3.4.2]{SAG}), we have \(R^+\simeq R\otimes \Z\) and one can deduce that \(\fib(\gamma)\simeq R\otimes \tau_{\geq 1}\S\) as \(\gamma\) is a left inverse to \(R\to R\otimes \Z\). Since \(R\) is a \(\Z[p_{\leq n}^{-1}]\)-algebra, we get, by Serre's result on torsion in \(\pi_*\S\), that \(\fib(\gamma)\) is \((2p_{n+1}-3)\)-connective and its \(\pi_{2p_{n+1}-3}\) is \(p_{n+1}\)-torsion. The result then follows because \(L_{R/\Z}\) is connective.
\end{proof}

\begin{proof}[Proof of Proposition \ref{prop_anvs}]
Without loss of generality we may assume that \(R\) is a \(\Z[p_{\leq n}^{-1}]\)-algebra as well. We proceed by induction on \(m\). If \(m=0\), then \(S\) is discrete and the result is obvious.

The animated ring \(S\) is a square-zero extension of \(\tau_{<m}S\) by \(\pi_mS[m]\). This is classified by an \(S\)-linear map \(L^\alg_{\tau_{< m}S}\to\pi_mS[m+1]\). We then get a commutative diagram
\[\begin{tikzcd}
\Map_\an(R,S) \ar[r] \ar[d] & \Map_\an(R,\tau_{<m}S) \ar[r] \ar[d] & \Map_R(L^\alg_R,\pi_mS[m+1]) \ar[d] \\
\Map_\Z(R,S) \ar[r] & \Map_\Z(R,\tau_{<m}S) \ar[r] & \Map_R(L_{R/\Z},\pi_mS[m+1])
\end{tikzcd}\]
where both rows are fiber sequences over the zero map. The middle vertical map is an equivalence by the induction hypothesis. By Lemma \ref{lem_cot}, the fiber of the right vertical map is contractible if either \(m<2p_{n+1}-3\) or \(\pi_{2p_{n+1}-3}S\) is \(p_{n+1}\)-torsion free. It is discrete otherwise. The result follows.
\end{proof}

\begin{acor} \label{cor_2trun}
Let \(R\) and \(S\) be animated rings. Assume that \(S\) is 2-truncated. Then the comparison map
\[\Map_\an(R,S)\to\Map_\Z(R,S)\]
has contractible fibers. If moreover \(\pi_1S\) is 2-torsion free, then it is an equivalence.
\end{acor}

\begin{acor}
Let \(0\leq n\leq \infty\). The forgetful functor
\[\Ring^\an_{\Z[p_{\leq n}^{-1}]}\to\CAlg^\cn_{\Z[p_{\leq n}^{-1}]}\]
induces equivalences on mapping spaces after \((2p_{n+1}-4)\)-truncation.
\end{acor}

\begin{arem}
Inspecting the proofs we also see that if \(R\) is étale over a base animated ring \(k\), then \(\Map_{\Ring^\an_k}(R,S)\to\Map_{\CAlg^\cn_k}(R,S)\) is an equivalence without any assumptions on \(S\).
\end{arem}

\section{\(\delta\)-rings and the cotangent complex} \label{sec_cotan}

In this section, we use the cotangent complex formalism of \cite{HA} and \cite{SAG} to study \(\delta\)-rings.

Recall that a \(\delta_p\)-ring is an animated ring \(R\) together with a lift of the Frobenius \(R\to R/p\). That is, a \(\delta_p\)-structure on \(R\) is an endomorphism \(\phi:R\to R\) together with a homotopy making the diagram
\[\begin{tikzcd}
& R \ar[d,"\can"] \\ R \ar[r,"\Frob"] \ar[ur,"\phi"] & R/p
\end{tikzcd}\]
commute.

Let \(R\) be an \(n\)-truncated animated ring. Then \(R/p\) is \((n+1)\)-truncated and its \(\pi_{n+1}\) can be identified with the \(p\)-torsion submodule \(\pi_nR[p]\) of \(\pi_nR\). Motivated by the discrete case, we say that a \emph{non-derived Frobenius lift} is an endomorphism \(\phi:R\to R\) together with a homotopy making the diagram
\[\begin{tikzcd}
& R \ar[d,"\can"] \\[8] R \ar[r,"\Frob"] \ar[ur,"\phi"] & \tau_{\leq n}(R/p)
\end{tikzcd}\]
commute.

Our goal is to establish what is needed in order to upgrade a non-derived Frobenius lift to a \(\delta_p\)-structure.

\begin{acon} \label{con_obmap}
Let \(R\) be an \(n\)-truncated animated ring and let \(\phi:R\to R\) be a non-derived Frobenius lift. We then have the following diagram
\[\begin{tikzcd}
& L^\alg_R \ar[d,"L^\alg_\can"] \ar[ddr,"0"{name=B}, bend left] \ar[Rightarrow, from=2-2, to=B, shorten <=2mm, shorten >=6mm,"\beta" pos=0.4] \\[8]
L^\alg_R \ar[ur,"L^\alg_\phi"] \ar[r,"L^\alg_\Frob"] \ar[drr,"0"'{name=R}, bend right] \ar[Rightarrow, from=2-2, to=R, shorten <=3mm, shorten >=5mm,"\alpha" pos=0.3] & L^\alg_{\tau_{\leq n}(R/p)} \ar[dr] \\
&& (\pi_nR[p])[n+2]
\end{tikzcd}\]
where the nulhomotopies \(\alpha\) and \(\beta\) are induced respectively by the Frobenius and by the quotient map \(R \to R/p\). Composing \(\alpha\) with the inverse of \(\beta\) yields a self-homotopy of the zero map \(L^\alg_R\to (\pi_nR[p])[n+2]\), hence a map
\[L^\alg_R\to (\pi_nR[p])[n+1]\]
of \(R\)-modules, where \(\pi_nR[p]\) is equipped with the \emph{Frobenius twisted} \(R\)-module structure (restriction of scalars of the canonical module structure along \(\Frob\)).
\end{acon}

The cotangent complex formalism then yields the following.

\begin{acor} \label{cor_obmap}
Let \(R\) be an \(n\)-truncated animated ring and let \(\phi:R\to R\) be a non-derived Frobenius lift. Then the space of nulhomotopies for \(L_R^\alg\to (\pi_nR[p])[n+1]\) from Construction \ref{con_obmap} is equivalent to the space of \(\delta_p\)-structures on \(R\) with \(\phi\) the underlying non-derived Frobenius lift.

Moreover, a map between \(n\)-truncated \(\delta_p\)-rings \(R\to S\) is equivalent to a map \(R\to S\) compatible with non-derived Frobenius lifts and a homotopy between the two induced nulhomotopies of \(L_R^\alg \to (\pi_nS[p])[n+1]\).
\end{acor}

\begin{arem}
We also see that if a non-derived Frobenius lift on \(R\) extends to a \(\delta_p\)-structure, then the space of extensions forms a torsor over \(\Map_R\big(L_R^\alg,(\pi_nR[p])[n]\big)\). If \(R\) is a discrete \(\delta_p\)-ring, we obtain that the set of functions \(\delta:R\to R\) inducing the same endomorphism \(\phi\) is a torsor over \(\Der_R(R,R[p])\). Indeed, given two such functions, one can check that their difference is a derivation into \(R[p]\) (with the Frobenius twisted \(R\)-module structure).
\end{arem}

As an application for the analysis above we give an alternative proof of the following result. The only other proof we know uses Witt vectors and follows from the results of \cite[Appendix A]{Prismatization}.

\begin{aprop} \label{prop_trunpdelta}
Let \(R\) be an animated \(\delta_p\)-ring. Then the truncation \(\tau_{\leq n}R\) is naturally a \(\delta_p\)-ring and the canonical map \(R\to\tau_{\leq n}R\) refines to a \(\delta_p\)-map.
\end{aprop}

The main input in our proof is the following lemma.

\begin{alem} \label{lem_cottau}
Let \(R\) be an animated ring. Then the fiber of the canonical map \(L^\alg_{R}\to L^\alg_{\tau_{\leq n}R}\) is \((n+1)\)-connective. In particular, for any \((n+1)\)-truncated \(X\in\Mod_R\), the map \(\Map_R(L^\alg_{\tau_{\leq n}R},X)\to\Map_R(L_R^\alg,X)\) has contractible fibers.
\end{alem}

\begin{proof}
By \cite[Corollary 25.3.6.4]{SAG}, the relative cotangent complex \(L^\alg_{\tau_{\leq n}R/R}\) is \((n+2)\)-connective. From the fiber sequence \(L_R^\alg\otimes_R\tau_{\leq n}R\to L_{\tau_{\leq n}R}^\alg\to L^\alg_{\tau_{\leq n}R/R}\) we deduce that the fiber of the map \(L_R^\alg\otimes_R\tau_{\leq n}R\to L^\alg_{\tau_{\leq n}R}\) is \((n+1)\)-connective. On the other hand, from the fiber sequence \(L_R^\alg\otimes_R\tau_{>n}R\to L_R^\alg \to L_R^\alg\otimes_R\tau_{\leq n}R\) we deduce that the fiber of \(L_R^\alg\to L_R^\alg\otimes_R\tau_{\leq n}R\) is \((n+1)\)-connective, and so the same holds for their composite \(L_R^\alg\to L_{\tau_{\leq n}R}^\alg\).
\end{proof}

\begin{proof}[Proof of Proposition \ref{prop_trunpdelta}]
Let \(\phi:R\to R\) denote the Frobenius lift for \(R\). Applying \(\tau_{\leq n}\) we get a commutative diagram
\[\begin{tikzcd}
& \tau_{\leq n} R \ar[d,"\can"] \\[8]
\tau_{\leq n}R \ar[ur,"\tau_{\leq n}\phi"] \ar[r,"\Frob"] & \tau_{\leq n}(R/p),
\end{tikzcd}\]
hence a non-derived Frobenius lift for \(\tau_{\leq n}R\). By Corollary \ref{cor_obmap}, we get an obstruction map \(L^\alg_{\tau_{\leq n}R}\to (\pi_nR[p])[n+1]\) which we want to show nulhomotopic. Since we have a commutative diagram
\[\begin{tikzcd}
& R \ar[d,"\can"] \\[8]
R \ar[ur,"\phi"] \ar[r,"\Frob"] & (\tau_{\leq n}R)/p,
\end{tikzcd}\]
we get a nulhomotopy for the composite \(L_R^\alg \to L_{\tau_{\leq n}R}^\alg \to (\pi_nR[p])[n+1]\). By Lemma \ref{lem_cottau} the map
\[\Map_R\big(L_{\tau_{\leq n}R}^\alg,(\pi_nR[p])[n+1]\big)\to \Map_R\big(L_R^\alg,(\pi_nR[p])[n+1]\big)\]
has contractible fibers, so we get a nulhomotopy for \(L_{\tau_{\leq n}R}^\alg\to (\pi_nR[p])[n+1]\) as well. In order to refine \(R\to\tau_{\leq n}R\) to a \(\delta_p\)-map, we need to supply a homotopy between the nulhomotopy for \(L_R^\alg \to (\pi_nR[p])[n+1]\) and the composite of \(L_R^\alg\to L_{\tau_{\leq n}R}^\alg\) with the nulhomotopy for \(L_{\tau_{\leq n}R}^\alg \to (\pi_nR[p])[n+1]\). This is automatic by construction.
\end{proof}

\begin{arem}
Let \(S\) and \(R\) be animated \(\delta\)-rings. Observe that a \(\delta\)-ring map \(S\to \tau_{\leq n}R\) can be obtained from a \(\delta\)-ring map \(S\to \tau_{\leq n-1}R\) in two steps. First by extending it to a map \(S\to \tau_{\leq n}R\) of rings with non-derived Frobenius lifts and then upgrading it to a map of \(\delta\)-rings. Both of them only involve lifting along square-zero extensions, so can be analysed with our approach.

For example, the free discrete \(\delta\)-ring \(\Z\{x\}\) is a polynomial algebra, so has a very simple cotangent complex. One can then compute the mapping space \(\Map_{\delta}(\Z\{x\},R)\) for any animated \(\delta\)-ring using this strategy to verify that \(\Z\{x\}\) is also the free animated \(\delta\)-ring. It then formally follows that the category of animated \(\delta\)-rings is the animation of \(\delta\)-rings.
\end{arem}

\printbibliography

\end{document}